\pdfoutput=1

\documentclass[a4paper, twoside, article,english,leqno,10pt]{memoir}

\usepackage[utf8]{inputenc} %for using Unicode
\usepackage[T1]{fontenc} 	%European fonts
\usepackage{babel}
\usepackage[noDcommand,slantedGreeks]{kpfonts}

\title{Equivariant sheaves on loop spaces}
\date{\today}
\author{Sergey Arkhipov and Sebastian Ørsted}

\usepackage{mathtools, microtype,xspace}

\usepackage[shortlabels]{enumitem}

\usepackage[draft]{fixme}

\nouppercaseheads %running heads should not be capitalized
\captionnamefont{\small} %captions with small font
\captiontitlefont{\small}
\makeevenhead{headings}{\thepage}{}{\itshape\leftmark} %make headings italic instead of slanted
\makeoddhead{headings}{\itshape\rightmark}{}{\thepage}

\raggedbottomsectiontrue%less harse than \raggedbottom
\makeatletter
\g@addto@macro\bfseries{\boldmath} %make math in bold text automatically bold
\makeatother

\usepackage[autostyle,german=guillemets,english=american]{csquotes}

\usepackage{tikz}

\usetikzlibrary{babel}

\usetikzlibrary{cd}

\tikzcdset{
  arrow style=tikz,
}

\tikzset{
  >/.tip={Stealth[length=2.9pt, width=4.4pt, inset=1.8pt]},
  tikzcd left hook/.tip={Hooks[
	  left,
	  length=2pt,
	  width=5.5pt,
	]},
  iso/.style={
    every to/.append style={
      edge node={
        node [sloped, allow upside down]{
          \raisebox{0.1em}[0pt][0pt]{\ensuremath{\sim}}
        }
      }
    }
  },
  iso'/.style={
    every to/.append style={
      edge node={
        node [sloped, allow upside down]{
          \raisebox{-0.6em}[0pt][0pt]{\ensuremath{\sim}}
        }
      }
    }
  },
  symbol/.style={
      draw=none,
      every to/.append style={
        edge node={node [sloped, allow upside down, auto=false]{$#1$}}}
  },
}

\newcommand\invdot{\mathrlap{.}}
\newcommand\invcomma{\mathrlap{,}}

\usetikzlibrary{patterns}
\usetikzlibrary{decorations.pathreplacing,decorations.markings}

\tikzset{
  on each segment/.style={
    decorate,
    decoration={
      show path construction,
      moveto code={},
      lineto code={
        \path [#1]
        (\tikzinputsegmentfirst) -- (\tikzinputsegmentlast);
      },
      curveto code={
        \path [#1] (\tikzinputsegmentfirst)
        .. controls
        (\tikzinputsegmentsupporta) and (\tikzinputsegmentsupportb)
        ..
        (\tikzinputsegmentlast);
      },
      closepath code={
        \path [#1]
        (\tikzinputsegmentfirst) -- (\tikzinputsegmentlast);
      },
    },
  },
  mid arrow/.style={postaction={decorate,decoration={
        markings,
        mark=at position .5 with {\arrow[#1]{stealth}}
      }}},
}

\usepackage[backend=biber,citestyle=authoryear-ibid,bibstyle=authoryear,
	uniquename=init,giveninits=true,autolang=hyphen]{biblatex}
\nobibintoc %% Removes bibliography from the pable of contents.

\usepackage[thmmarks,amsmath]{ntheorem}

\makeatletter

\makeatletter
\newtheoremstyle{notitle}% for exercises
{\item[\hskip\labelsep \theorem@headerfont ##2\theorem@separator]}%
{\item[\hskip\labelsep \theorem@headerfont ##2\theorem@separator\ ##3\theorem@separator]}

\newtheoremstyle{notitlebreak}% for exercises with a break
{\item[\rlap{\vbox{\hbox{\hskip\labelsep \theorem@headerfont ##2\theorem@separator}\hbox{\strut}}}]}%
{\item[\rlap{\vbox{\hbox{\hskip\labelsep \theorem@headerfont ##2\theorem@separator\ ##3\theorem@separator}\hbox{\strut}}}]}

\newtheoremstyle{prenumber}%
  {\item[\hskip\labelsep \theorem@headerfont ##2\theorem@separator\ ##1\theorem@separator]}%
  {\item[\hskip\labelsep \theorem@headerfont ##2\theorem@separator\ ##3\theorem@separator]}

\newtheoremstyle{prenumberbreak}%
  {\item[\rlap{\vbox{\hbox{\hskip\labelsep \theorem@headerfont 
          ##2\theorem@separator\ ##1\theorem@separator}\hbox{\strut}}}]}%
  {\item[\rlap{\vbox{\hbox{\hskip\labelsep \theorem@headerfont 
          ##2\theorem@separator\ ##3\theorem@separator}\hbox{\strut}}}]}

\newtheoremstyle{customplain}%
  {\item[\hskip\labelsep \theorem@headerfont ##1\ ##2\theorem@separator]}%
  {\item[\hskip\labelsep \theorem@headerfont ##3\ ##2\theorem@separator]}
  
\newtheoremstyle{custombreak}%
  {\item[\rlap{\vbox{\hbox{\hskip\labelsep \theorem@headerfont
          ##1\ ##2\theorem@separator}\hbox{\strut}}}]}%
  {\item[\rlap{\vbox{\hbox{\hskip\labelsep \theorem@headerfont 
          ##3\ ##2\theorem@separator}\hbox{\strut}}}]}

\newtheoremstyle{postnumpara}%
  {\item[\hskip\labelsep \theorem@headerfont ##2\theorem@separator]}%
  {\item[\hskip\labelsep \theorem@headerfont ##3\ ##2\theorem@separator]}
  
\newtheoremstyle{postnumparabreak}%
  {\item[\rlap{\vbox{\hbox{\hskip\labelsep \theorem@headerfont
          ##2\theorem@separator}\hbox{\strut}}}]}%
  {\item[\rlap{\vbox{\hbox{\hskip\labelsep \theorem@headerfont 
          ##3\ ##2\theorem@separator}\hbox{\strut}}}]}

\newtheoremstyle{nonumberplainflex}% for proof environments, see below
	{\item[\theorem@headerfont\hskip\labelsep ##1\theorem@separator]}%
	{\item[\theorem@headerfont\hskip \labelsep ##3\theorem@separator]}

\newtheoremstyle{nonumberbreakflex}%
	{\item[\rlap{\vbox{\hbox{\hskip\labelsep \theorem@headerfont
		##1\theorem@separator}\hbox{\strut}}}]}%
	{\item[\rlap{\vbox{\hbox{\hskip\labelsep \theorem@headerfont 
		##3\theorem@separator}\hbox{\strut}}}]}

\makeatother

\theoremseparator{.}

\theorembodyfont{\normalfont}

\theoremstyle{notitle}

\theoremstyle{prenumber}

\theorembodyfont{\itshape}

\newtheorem{theorem}[equation]{Theorem}
\newtheorem{corollary}[equation]{Corollary}
\newtheorem{proposition}[equation]{Proposition}
\newtheorem{lemma}[equation]{Lemma}

\theorembodyfont{\normalfont}
\theoremsymbol{\ensuremath{\triangle}}
\newtheorem{remark}[equation]{Remark}

\theoremsymbol{\ensuremath{\scriptstyle\bigcirc}}
\newtheorem{example}[equation]{Example}

\theoremstyle{nonumberplainflex}
\theoremsymbol{\ensuremath{\square}}
\theoremheaderfont{\itshape\bfseries}
\newtheorem{proof}{Proof}

\theoremstyle{nonumberplainflex}
\theoremheaderfont{\normalfont\bfseries}
\theoremsymbol{}

\theorembodyfont{\normalfont}

\theorembodyfont{\itshape}
\newtheorem{theoremnonumber}{Theorem}
\newtheorem{corollarynonumber}{Corollary}

\theorembodyfont{\normalfont}
\theoremsymbol{\ensuremath{\triangle}}

\theoremsymbol{\ensuremath{\scriptstyle\bigcirc}}

\theoremsymbol{}
\theoremheaderfont{\normalfont\bfseries}

\theoremstyle{notitlebreak}

\theoremstyle{prenumberbreak}

\theorembodyfont{\itshape}
\newtheorem{theorembreak}[equation]{Theorem}

\theorembodyfont{\normalfont}
\theoremsymbol{\ensuremath{\triangle}}

\theoremsymbol{\ensuremath{\scriptstyle\bigcirc}}

\theoremsymbol{}

\theoremstyle{nonumberbreakflex}
\theoremsymbol{\ensuremath{\square}}
\theoremheaderfont{\itshape\bfseries}

\theoremstyle{nonumberbreakflex}
\theoremheaderfont{\normalfont\bfseries}
\theoremsymbol{}
\theorembodyfont{\normalfont}

\theorembodyfont{\itshape}

\theorembodyfont{\normalfont}
\theoremsymbol{\ensuremath{\triangle}}

\theoremsymbol{\ensuremath{\scriptstyle\bigcirc}}

\usepackage[hidelinks]{hyperref}
	
\hypersetup{
	pdfauthor={Sergey Arkhipov and Sebastian Ørsted},
	pdftitle={Equivariant sheaves on loop spaces},
}
	
\usepackage[nameinlink]{cleveref}

\crefname{para}{}{}
\crefname{theorem}{Theorem}{Theorems}
\crefname{definition}{Definition}{Definitions}
\crefname{example}{Example}{Examples}
\crefname{corollary}{Corollary}{Corollaries}
\crefname{proposition}{Proposition}{Propositions}
\crefname{lemma}{Lemma}{Lemmata}
\crefname{remark}{Remark}{Remarks}
\crefname{claim}{Claim}{Claims}
\crefname{exercise}{Exercise}{Exercises}
\crefname{textexercise}{Exercise}{Ecercises}

\crefalias{parabreak}{para}

\crefalias{theorembreak}{theorem}

\crefalias{definitionbreak}{definition}

\crefalias{examplebreak}{example}

\crefalias{corollarybreak}{corollary}

\crefalias{propositionbreak}{proposition}

\crefalias{lemmabreak}{lemma}

\crefalias{claimbreak}{claim}

\crefalias{remarkbreak}{remark}

\crefalias{textexercisebreak}{textexercise}

\newlist{paralist}{enumerate}{2}
\setlist[paralist]{
	label=\textup{(\roman*)},
	ref=\thepara\textup{(\roman*)},
	resume,
}
\crefalias{paralisti}{theorem}

\newlist{theoremlist}{enumerate}{2}
\setlist[theoremlist]{
	label=\textup{(\roman*)},
	ref=\thepara\textup{(\roman*)},
	resume,
}
\crefalias{theoremlisti}{theorem}

\newlist{corollarylist}{enumerate}{2}
\setlist[corollarylist]{
	label=\textup{(\roman*)},
	ref=\thepara\textup{(\roman*)},
	resume,
}
\crefalias{corollarylisti}{corollary}

\newlist{propositionlist}{enumerate}{2}
\setlist[propositionlist]{
	label=\textup{(\roman*)},
	ref=\thepara\textup{(\roman*)},
	resume,
}
\crefalias{propositionlisti}{proposition}

\newlist{lemmalist}{enumerate}{2}
\setlist[lemmalist]{
	label=\textup{(\roman*)},
	ref=\thepara\textup{(\roman*)},
	resume,
}
\crefalias{lemmalisti}{lemma}

\newlist{definitionlist}{enumerate}{2}
\setlist[definitionlist]{
	label=\textup{(\roman*)},
	ref=\thepara\textup{(\roman*)},
	resume,
}
\crefalias{definitionlisti}{definition}

\newlist{textexerciselist}{enumerate}{2}
\setlist[textexerciselist]{
	label=\textup{(\roman*)},
	ref=\thepara\textup{(\roman*)},
	resume,
}
\crefalias{textexerciselisti}{textexercise}

\newlist{remarklist}{enumerate}{2}
\setlist[remarklist]{
	label=\textup{(\roman*)},
	ref=\thepara\textup{(\roman*)},
	resume,
}
\crefalias{remarklisti}{remark}

\newlist{examplelist}{enumerate}{2}
\setlist[examplelist]{
	label=\textup{(\roman*)},
	ref=\thepara\textup{(\roman*)},
	resume,
}
\crefalias{examplelisti}{example}

\newlist{exerciselist}{enumerate}{2}
\setlist[exerciselist]{
	label=\textup{(\roman*)},
	ref=\thepara\textup{(\roman*)},
	resume,
}
\crefalias{exerciselisti}{exercise}

\newlist{claimlist}{enumerate}{2}
\setlist[claimlist]{
	label=\textup{(\roman*)},
	ref=\thepara\textup{(\roman*)},
	resume,
}
\crefalias{claimlisti}{claim}

\makeatletter
\newcounter{localreftmpcnt} %
\newcommand\alphsubformat[1]{\textup{(\roman{#1})}} %adapt ....
\newcommand\localref[2][\alphsubformat]{%
	\ifcsname r@#2@cref\endcsname
	\cref@getcounter {#2}{\mylabel}%
	\setcounter{localreftmpcnt}{\mylabel}%
	\hyperref[#2]{%
		\alphsubformat{localreftmpcnt}%
	}%
	\else ?? \fi}

\newcommand\isoto{\xrightarrow{ %isomorphism arrow
		\smash{\raisebox{-0.65ex}{\ensuremath{\scriptstyle\sim}}}}}

\newcommand\longisoto{\xrightarrow{ %isomorphism arrow
		\;\smash{\raisebox{-0.65ex}{\ensuremath{\scriptstyle\sim}}}\;}}

\providecommand*{\twoheadrightarrowfill@}{%
  \arrowfill@\relbar\relbar\twoheadrightarrow
}
\providecommand*{\twoheadleftarrowfill@}{%
  \arrowfill@\twoheadleftarrow\relbar\relbar
}
\providecommand*{\xtwoheadrightarrow}[2][]{%
  \ext@arrow 0579\twoheadrightarrowfill@{#1}{#2}%
}
\providecommand*{\xtwoheadleftarrow}[2][]{%
  \ext@arrow 5097\twoheadleftarrowfill@{#1}{#2}%
}

\newcommand\xto[1]{
	\xrightarrow{\smash{\raisebox{-.2em}[0pt][0pt]{$\scriptscriptstyle#1$}}}
}

\newcommand\longto\longrightarrow

\newcommand\into{\hookrightarrow}

\makeatother

\let\textdef=\textbf

\makeatother

\newcommand\slot{\mathord{-}}

\newcommand\smallbullet{\raisebox{-0.25ex}{\scalebox{1.2}{$\cdot$}}}

\newcommand\ainfty{\ensuremath{\mathup{A}_{\infty}}\xspace}

\newcommand\categoryformat[1]{\textup{#1}}

\newcommand\dirlimformat[1]{
		\mathop{
			\smash{
				\operatorname*{#1}\limits_{
					{}
					\raisebox{.21em}[0pt][0pt]{\scalebox{.85}{$
						\xrightarrow{%
							\hphantom
							{%
								\!\!\!\!
								\scalebox{1.17}{$
									{\operatorname{#1}}
								$}
								\!\!\!
							}
						}
					$}}
				}
			}
			\vphantom{\textstyle\lim_n}
		}
}

\newcommand\invlimformat[1]{
		\mathop{
			\smash{
				\operatorname*{#1}\limits_{
					{}
					\raisebox{.21em}[0pt][0pt]{\scalebox{.85}{$
						\xleftarrow{%
							\hphantom
							{%
								\!\!\!\!
								\scalebox{1.17}{$
									{\operatorname{#1}}
								$}
								\!\!\!
							}
						}
					$}}
				}
			}
			\vphantom{\textstyle\lim_n}
		}
}

\numberwithin{equation}{section}

\begin{document}

\maketitle

\begin{abstract}
	\noindent
	Let~\( X \) be an affine, smooth, and Noetherian scheme over~\( \mathbb{C} \)
	acted on by an affine algebraic group~\( G \). Applying the technique developed in \textcite{ends,comodules}, we define a dg-model for the derived category of dg-modules over the dg-algebra of differential forms~\( \Omega_{X} \) on~\( X \) equivariant with respect to the action of a derived group scheme \( (G ,\Omega_{G} ) \). We compare the obtained dg-category with the one considered in \textcite{tina} given by coherent sheaves on the derived Hamiltonian reduction of~\( T^{*} X  \).
\end{abstract}

\chapter{Introduction}

Let~\( G \) be an affine algebraic group acting on an affine, smooth, and Noetherian scheme~\( X \) over~\( \mathbb{C} \). In the recent paper \textcite{tina}, the first author jointly studied the category of quasi-coherent sheaves on the derived Hamiltonian reduction of~\( T^{*} X  \)
with respect to~\( G \). As stated in the introduction to that paper, this category is supposed to play the role of a category of equivariant sheaves on the derived loop space of the scheme~\( X \). The present paper is a step towards making this intuition into a precise construction.

Omitting equivariance, the picture is well understood. Namely, in
homotopy theory, the free loop space of an object~\( X \) (say, in a model category) is the homotopy fibre product of~\( X \) with itself over~\( X \times X \).
Thus, for regular schemes, the derived loop space of a scheme~\( X \) is given by the sheafified Hochschild homology dg-algebra for~\( \mathscr{O}_{X} \).

The famous Hochschild--Kostant--Rosenberg theorem states that the latter
is formal: it is quasi-isomorphic to its cohomology given by the sheaf of dg-algebras of differential forms on~\( X \) with zero differential. Denote it by~\( \Omega_{X} \). It follows that talking about sheaves on the derived loop space
of~\( X \) is essentially considering quasi-coherent sheaves of dg-modules
over~\( \Omega_{X} \).
One of the goals of the present paper is to make precise sense of the
derived category of \( \Omega_{X} \)-dg-modules equivariant with respect to the action of the derived group scheme~\( (G ,\Omega_{G} ) \).

Our strategy is as follows. We consider the simplicial derived scheme given by the nerve of the
action groupoid for~\( (G ,\Omega_{G} ) \) acting on~\( (X ,\Omega_{X} ) \). This way, we obtain a
cosimplicial diagram of dg-derived categories. We consider the homotopy totalization of this diagram. More precisely, we follow the strategy of~\textcite{explicit} generalized and extended in our
previous papers~\textcite{ends,comodules}. We obtain a category of \ainfty-comodules over a certain dg-coalgebra in the category of \( \Omega_{X} \)-dg-modules. This category becomes our model for the derived category of \( \Omega_{X} \)-dg-modules equivariant with respect to \( \Omega_{G} \).

Applying linear Koszul duality in the spirit of the papers of Mirkovic
and Riche, we pass from \( \Omega_{X} \)-dg-modules to the \enquote{even} side of the duality. We consider the sheaf of dg-algebras on~\( X \) given
by~\(
	\operatorname{Sym}_{\mathscr{O}_{X}} (T_{X} [{-2}] )
\)
and quasi-coherent dg-modules over it. We call the obtained category
the derived category of quasi-coherent sheaves on the \emph{homologically
shifted} cotangent bundle of~\( X \).

After applying Koszul duality, the
\( (G ,\Omega_{G} ) \)-equivariance becomes a more subtle structure.
Notice that \( (G ,\Omega_{G} ) \) contains~\( (G ,\mathscr{O}_{G} ) \) both as a subgroup and as a
quotient group. The Koszul duality construction respects the~\( (G ,\mathscr{O}_{G} ) \)-action.
The remaining coaction of the Hopf dg-algebra of left invariant differential forms on~\( G \) becomes the following structure. Denote the Lie
algebra of~\( G \) by~\( \mathfrak{g} \). The moment map for the \( G \)-action on~\( T^{*} X  \) provides a \( G \)-equivariant map~\(
	\mathfrak{g} \mathbin{\otimes} \mathscr{O}_{X} [{-1}]\to T_{X}
\). Consider the free graded
commutative algebra generated by the two term complex
\[
	\operatorname{Sym}_{\mathscr{O}_{X}} (\mathfrak{g} \mathbin{\otimes}{\mathscr{O}_{X} [{-1}]} \to{T_{X} [{-2}]} ).
\]
Notice that up to the homological shift, this dg-algebra is a model
for functions on the derived preimage of~\( 0 \) under the moment map \( \mu  \colon T^{*} X  \to \mathfrak{g}^{*} \). We come to the central statement of the present paper:

\begin{theoremnonumber}[\ref{res:main_theorem}. Theorem]
	Let~\( X \) ba an affine, smooth, and Noetherian scheme over~\( \mathbb{C} \) acted on by an affine group scheme~\( G \).
	The derived category of \( (G ,\Omega_{G} ) \)-equivariant \( \Omega_{X} \)-dg-modules is equivalent to the triangulated subcategory
	\[
		\langle \mathscr{O}_{X} \rangle
		\subset
		\mathscr{D}
		\bigl(
			\operatorname{Sym}_{\mathscr{O}_{X}} (\mathfrak{g} \mathbin{\otimes} \mathscr{O}_{X} [{-1}] \to{T_{X} [{-2}]} )
			\categoryformat{-dgmod}
		\bigr)^{G}
	\]
	of the derived category of \( G \)-equivariant dg-modules
	generated by \( \mathscr{O}_{X} \) and closed under small coproducts.
\end{theoremnonumber}

Notice that the case of~\( X \) equal to a point is of interest. In this case, the statement, combined with the usual Koszul duality between
\( \operatorname{Sym} (\mathfrak{g} [{-1}] ) \) and~\( \operatorname{Sym} (\mathfrak{g}^{*} ) \),
reads as follows:

\begin{corollarynonumber}
	The derived category of representations for the group~\( (G ,\Omega_{G} ) \) is equivalent to the derived category of \( G \)-equivariant quasicoherent sheaves on~\( \mathfrak{g} \) topologically supported at~\( 0 \).
\end{corollarynonumber}

\begingroup

\chapter{Derived categories and \texorpdfstring{\ainfty}{A∞}-modules}

In this chapter, we discuss and compare different candidates for derived categories of dg-modules over a dg-algebra~\( A \).

\section{The dg-derived category}

Let~\( k \) be a field of characteristic zero.
If \( A \)~is a dg-algebra over~\( k \),
recall that its \textdef{derived dg-category}
is defined to be the Drinfeld localization~\(
	\mathscr{D} (\categoryformat{$A $-dgmod} )
	\cong
	\categoryformat{$A $-dgmod} [{W^{-1}}]
\)
at the class~\( W \) of quasi-isomorphisms
\parencite[see][]{drinfeld}.
Its associated homotopy category~\( 
	H^{0} \mathscr{D} (\categoryformat{$A $-dgmod} )
\)
recovers the conventional derived category
obtained via the machinery of triangulated categories.

Let~\( M \) be a dg-module over the dg-algebra~\( A \).
We say that it is \textdef{graded projective} if it is projective as a graded module over the graded \( k \)-algebra~\( A \).
The full dg-subcategory of graded projective modules is denoted~\( \categoryformat{Proj} (A ) \subset \categoryformat{$A $-dgmod} \).
The dg-module~\( M \) is called~\textdef{homotopy-projective} (or~simply~\textdef{h-projective})
if
for any exact dg-module~\( X \) over~\( A \),
the \( k \)-complex~\( \operatorname{Hom}_{A}^{\smallbullet} (M ,X ) \)
is exact.
The full dg-subcategory of such is denoted~\( \categoryformat{H-Proj} (A ) \subset \categoryformat{$A $-dgmod} \).
A dg-module~\( M \) is called \textdef{semifree}
if it admits an ascending, bounded below, exhaustive filtration~\(
	0
	=
	F^{0} M
	\subset
	F^{1} M
	\subset
	F^{2} M
	\subset\cdots\subset
	M
\)
such that each graded piece~\( \mathup{gr}^{n} M \)
is the direct sum of shifts of copies of~\( A \).
We denote the full subcategory of such
by~\( \categoryformat{SF} (A ) \subset \categoryformat{$A $-dgmod} \).
The dg-subcategory of \textdef{quasi-free}
dg-modules is defined to be the dg-subcategory~\(
	\categoryformat{QF} (A ) = \categoryformat{Proj} (A ) \cap \categoryformat{H-Proj} (A )
\)
of dg-modules which are simultaneously graded projective and h-projective.
Note that any semifree dg-modules is quasi-free, so \( \categoryformat{SF} (A ) \subset \categoryformat{QF} (A ) \).

\begin{proposition}
	Each map in the composition
	\[
		\categoryformat{SF} (A )
		\subset
		\categoryformat{QF} (A )
		\subset
		\categoryformat{H-Proj} (A )
		\to
		\mathscr{D} (\categoryformat{$A $-dgmod} )
	\]
	is a quasi-equivalence of dg-categories.
	In particular, the first three all present the dg-derived category of~\( \categoryformat{$A $-dgmod} \).
\end{proposition}

\begin{proof}
	The fact that \( \categoryformat{SF} (A ) \) and~\( \categoryformat{H-Proj} (A ) \)~present the derived
	category is classical, see e.g. \textcite{drinfeld}. Since \( \categoryformat{QF} (A ) \)~sits
	between them as a full dg-subcategory, this implies that it, too,
	presents~\( \mathscr{D} (\categoryformat{$A $-dgmod} ) \).
\end{proof}

Suppose that \( R \)~is another dg-algebra over~\( k \).
A \textdef{dg-algebra} over~\( R \) is a dg-algebra~\( A \)
over~\( k \) together with a map of \( k \)-dg-algebras~\( R \to A \) (note that we do \emph{not} assume that its image is contained in the centre in any way).
Note that this is equivalent to~\( A \) being a unital algebra object in the monoidal dg-category~\( (\categoryformat{$R$-mod-$R$},\mathbin{\otimes_{R}} ) \).
A map of dg-algebras over~\( R \) is a map preserving this structure.

\begin{lemma}\label{lemma:restriction_preserves_H-Proj}
	Let~\( A \) be a dg-algebra over the dg-algebra~\( R \) which is projective (resp.~h-projective resp.~quasi-free) as a left \( R \)-dg-module.
	Then restriction
	\[
		\categoryformat{$A $-dgmod}
		\to
		\categoryformat{$R $-dgmod}
	\]
	takes projective (resp.~h-projective resp.~quasi-free)
	\( A \)-dg-modules to
	projective (resp.~h-projective resp.~quasi-free)
	\( R \)-dg-modules.
\end{lemma}

\begin{proof}
	For \( M \in \categoryformat{$A $-dgmod} \), this follows from the
	observation that
	\[
		\operatorname{Hom}_{R} (M|_{R},\slot )
		\cong
		\operatorname{Hom}_{A} (M ,\operatorname{Hom}_{R} (A ,\slot ) ).
	\]
\end{proof}

\begin{lemma}\label{res:tensors_of_quasi-free}
	If \( A \)~is a dg-algebra, then
	tensor products of
	projective (resp.~h-projective resp.~quasi-free) modules
	are
	projective (resp.~h-projective resp.~quasi-free).
\end{lemma}

\begin{proof}
	Apply the adjunction statement
	\[
		\operatorname{Hom}_{A} (M \mathbin{\otimes_{A}} N ,\slot )
		\cong
		\operatorname{Hom}_{A} (M ,\operatorname{Hom}_{A} (N ,\slot ) )
		.
	\]
\end{proof}

\begin{lemma}\label{res:scalar_extension_of_quasi-free}
	If \( A \to B \)~is a map of dg-algebras,
	then scalar extension of projective (resp.~h-projective resp.~quasi-free) \( B \)-dg-modules produces
	projective (resp.~h-projective resp.~quasi-free) \( A \)-dg-modules.
\end{lemma}

\begin{proof}
	Follows from~\(
		\operatorname{Hom}_{A} (A \mathbin{\otimes_{B}} M ,N )
		\cong
		\operatorname{Hom}_{B} (M ,N|_{B} )
	\).
\end{proof}

Our main tool will be the dg-category~\( \categoryformat{QF} (A ) \).
The convenience of working with this comes from the following observation:

\begin{proposition}
	The dg-category~\( \categoryformat{QF} (A ) \)
	is closed under \textdef{filtered extensions}.
	This means that if~\( M \in \categoryformat{$A $-dgmod} \)
	admits a bounded below and exhaustive filtration~\(
		0 = F^{0} M
		\subset
		F^{1} M
		\subset
		F^{2} M
		\subset
		\cdots
		\subset
		M = \dirlimformat{lim}{F^{n} M}
	\)
	such that each graded piece~\( \mathup{gr}^{n} M \in \categoryformat{QF} (A ) \),
	then \( M \in \categoryformat{QF} (A ) \) as well.
\end{proposition}

\begin{proof}
	Each~\( F^{n} M \) is graded projective,
	so the inclusions~\( F^{n-1} M \into F^{n} M \)
	split as maps of graded modules.
	This means that \( M \cong \mathup{gr}^{\smallbullet} M \)~as graded modules, so as a graded module, \( M \)~is the direct sum of projectives and is hence graded projecitve.
	To prove that it is also h-projective,
	we let~\( X \) be an exact dg-modules over~\( A \).
	The short exact sequence~\(
		0
		\to
		F^{n-1} M
		\to
		F^{n} M
		\to
		\mathup{gr}^{n} M
		\to
		0
	\)
	splits as a short exact sequence of graded modules. Applying the graded hom, we obtain that the sequence
	\[
		0
		\longto
		\operatorname{Hom}_{A}^{\smallbullet} (\mathup{gr}^{n} M,X )
		\longto
		\operatorname{Hom}_{A}^{\smallbullet} (F^{n} M,X )
		\longto
		\operatorname{Hom}_{A}^{\smallbullet} (F^{n-1} M,X )
		\longto
		0
	\]
	is also exact. Now the first term is exact by assumption, while the last term is exact by induction. Therefore,
	\( \smash{\operatorname{Hom}_{A}^{\smallbullet} (F^{n} M,X )} \)~is exact.
	The sequence also shows that the map~\( \smash{
		\operatorname{Hom}_{A}^{\smallbullet} (F^{n} M,X )
		\to
		\operatorname{Hom}_{A}^{\smallbullet} (F^{n-1} M,X )
	} \)
	is surjective, so the
	inverse system~\( \smash{\operatorname{Hom}_{A}^{\smallbullet} (F^{n} M,X )} \)
	satisfies Mittag--Leffler.
	Therefore,
	\(
		\operatorname{Hom}_{A}^{\smallbullet} (M ,X )
		=
		\invlimformat{lim}{\operatorname{Hom}_{A}^{\smallbullet} (F^{n} M,X )}
	\)
	is exact, so \( M \)~is h-projective.
	(Alternatively, if one wants to avoid Mittag--Leffler,
	one may write~\( M \) as the the last term of a short exact sequence~\(
		0
		\longto
		\bigoplus F^{n} M
		\xto{\mathup{id}-s}
		\bigoplus F^{n} M
		\longto
		M
		\longto
		0
	\)
	where \( s \)~is the sum of the embeddings~\( F^{n} M \into F^{n+1} M \).
	By graded projectivity, this sequence splits as a short exact sequence of graded modules, so as before, the sequence~\(
		0
		\to
		\operatorname{Hom}_{A}^{\smallbullet} (M ,X )
		\to
		\operatorname{Hom}_{A}^{\smallbullet} (\bigoplus F^{n} M,X )
		\to
		\operatorname{Hom}_{A}^{\smallbullet} (\bigoplus F^{n} M,X )
		\to
		0
	\)~is exact.
	As the last two terms are exact complexes, so is the first.)
\end{proof}

\section{The dg-category of \texorpdfstring{\ainfty}{A∞}-modules}

We continue to work over a base which is a dg-algebra~\( R \)
over a field~\( k \).
Let
\[
	\operatorname{Bar} (A)
	=
	\bigoplus_{n \ge 0}{A[1]^{\otimes n}}
	=
	\bigoplus_{n \ge 0} A \mathbin{\otimes_{R}} A \mathbin{\otimes_{R}} \cdots \mathbin{\otimes_{R}} A [n]
\]
be the \textdef{bar construction} of~\( A \),
the coalgebra with cofree comultiplication
\[
	\Delta (a_{1} \mathbin{\otimes} a_{2} \mathbin{\otimes} \cdots \mathbin{\otimes} a_{n} )
	=
	\sum_{i=1}^{n-1}
		(a_{1} \mathbin{\otimes} \cdots \mathbin{\otimes} a_{i} )
		\mathbin{\otimes}
		(a_{i +1} \mathbin{\otimes} \cdots \mathbin{\otimes} a_{n} )
	.
\]
There is another natural candidate for a derived category of~\( \categoryformat{$A $-dgmod} \), obtained using the notion of an \textdef{\ainfty-module}.
An \ainfty-module over a dg-algebra~\( A \)
is a dg-comodule over~\smash{\( \operatorname{Bar} (A ) \)}
whose underlying graded comodule is cofree.
In other words, it has the form~\(
	\operatorname{Bar}_{A} (M ) = \operatorname{Bar} (A ) \mathbin{\otimes_{R}} M
\)
for some graded \( R \)-module~\( M \).
By cofreeness, the differential~\(
	d
	\colon
	\operatorname{Bar}_{A} (M )
	\to
	\operatorname{Bar}_{A} (M )[{1}]
\)
is given by a collection of maps~\(
	\mathup{ac}_{n}
	\colon
	A^{\otimes (n-1)}
	\mathbin{\otimes_{R}}
	M
	\to
	M[{2-n}]
\)
for~\( n \ge 1 \).
each of which we shall refer to as the \( n \)th~action map.
A map~\( f \colon M \to N \) of \ainfty-modules is a map of the corresponding dg-comodules,
and it boils down
to a collection of maps~\( f_{n} \colon A^{\otimes n-1} \mathbin{\otimes_{R}} M \to N[{1-n}] \) for all~\( n\ge 1 \)
satisfying a technical condition corresponding to~\( f \)
commuting with~\( d \).
We write~\( \categoryformat{$A $-mod}_{\infty}^{\categoryformat{nu}} \)
for the dg-category of (non-uniltal) \ainfty-modules over~\( A \).
This is a full dg-subcategory
of the dg-category of dg-comodules over the bar construction.
We may write~\( \categoryformat{$A $-mod}_{\infty}^{\categoryformat{nu}} (\categoryformat{$R $-dgmod} ) \) if we want
to stress that the \ainfty-modules are \( R \)-linear.

The cohomology~\( H^{\smallbullet} (M ) \)
with respect to the differential~\( d = \mathup{ac}_{1} \)
is a graded module over the graded ring~\( H^{\smallbullet} (A ) \).
The \ainfty-module~\( M \)~is called \textdef{homotopy-unital}
if \( H^{\smallbullet} (M ) \)~is unital over~\( H^{\smallbullet} (A ) \).
The full subcategory consisting of~\( \categoryformat{$A $-mod}_{\infty}^{\categoryformat{nu}} \)
of homotopy-unital \ainfty-modules
is denoted~\( \categoryformat{$A $-mod}_{\infty}^{\categoryformat{hu}} \).

An \ainfty-module \( M \)~is called \textdef{strictly unital}
if we have \smash{\( \mathup{ac}\circ(\eta \mathbin{\otimes}\mathup{id}_{M}) = \mathup{id}_{M} \)}
and \smash{\(
	\mathup{ac}_{n}
	\circ
	(
		\mathup{id}_{A}^{\otimes i}
		\mathbin{\otimes}
		\eta 
		\mathbin{\otimes}
		\mathup{id}_{A}^{\otimes j}
		\mathbin{\otimes}
		\mathup{id}_{M}
	)
	= 0
\)}
for all~\( n\neq2 \)
and all~\( i , j \)
with~\( i + j + 2 = n \)
(here, \( \eta \colon R\to A \)
is the unit map).
If \( A \)~is augmented, this
is the same as a dg-comodule over the augmented bar construction~\( \operatorname{Bar}_{R}^{+} (A ) = \smash{\operatorname{Bar}_{R} (\overline{A} )} \)
whose underlying graded comodule is cofree.
We write the corresponding dg-comodule
as~\( \operatorname{Bar}_{R}^{+} (M ) \).
A strictly unital \ainfty-module is in particular homotopy-unital.
A dg-module~\( M \) is a strictly unital \ainfty-module
with~\( \mathup{ac}_{2} = \mathup{ac} \)
and~\( \mathup{ac}_{1} = d_{M} \).
A map \( f \colon M \to N \)~of
strictly unital \ainfty-modules
is a map of \ainfty-modules
such that \( f_{1} \)~commutes
with the unit, and such that
\( \smash{
	f_{n}
	\circ
	(\mathup{id}^{\otimes i}\mathbin{\otimes}\eta \mathbin{\otimes}\mathup{id}^{\otimes j}\mathbin{\otimes}\mathup{id}_{M})
	=
	0
} \)
for~\( n>1 \) and~\( i + j + 2 = n \).
The non-full dg-subcategory of~\( \categoryformat{$A $-mod}_{\infty}^{\categoryformat{nu}} \) consisting of strictly unital \ainfty-modules
is denoted~\( \categoryformat{$A $-mod}_{\infty} \).

An \textdef{\ainfty-quasi-isomorphism}
is a morphism~\( f \colon M \to N \)
such that \( f_{1} \colon M \to N \)~is
a quasi-isomorphism of \( R \)-dg-modules,
where \( M \) and~\( N \)
are equipped with the differential~\( d = \mathup{ac}_{1} \).
Denote by
\[
	\categoryformat{$A $-mod}_{\infty} (\categoryformat{QF} (R ) )
	\subset
	\categoryformat{$A $-mod}_{\infty} (\categoryformat{$R $-dgmod} )
\]
the full subcategory of \ainfty-modules~\( M \)
such that \( (M ,\mathup{ac}_{1} ) \)~is quasi-free
as an \( R \)-dg-module (we use analogous notation for non-unital and homotopy-unital \ainfty-modules).

We shall need adjunction statements for modules and comodules.
If \( R \)~is a dg-algebra over~\( k \), \( A \) a dg-algebra over~\( R \),
and~\( C \) a dg-coalgebra over~\( R \), then
a \textdef{twisting cochain} is a 
closed map of \( R \)-dg-bimodules~\( \tau  \colon C \to A \)
of degree~\( 1 \) such that
\[
	d_{A} \tau 
	+
	\tau  d_{C}
	+
	m(\tau \mathbin{\otimes}\tau )\Delta 
	=
	0.
\]
If~\( M \in \categoryformat{$C $-dgcomod} \), we denote
by~\( A \mathbin{\otimes_{R}^{\tau }} M \in \categoryformat{$A $-dgmod} \)
the dg-module whose underlying graded module is~\( A \mathbin{\otimes_{R}} M \), but whose differential is given by
\[
	d_{A \mathbin{\otimes_{R}^{\tau }} M}
	=
	d_{A \mathbin{\otimes_{R}} M}
	+
	d_{A \mathbin{\otimes_{R}} M}^{\tau }
\]
where \( d_{A \mathbin{\otimes_{R}} M}^{\tau } \)~is the composition
\[
	d_{A \mathbin{\otimes_{R}} M}^{\tau }
	\colon
	A \mathbin{\otimes_{R}} M
	\xrightarrow{\mathup{id}_{A}\mathbin{\otimes}\mathup{ca}}
	A\mathbin{\otimes}C\mathbin{\otimes}M
	\xrightarrow{\mathup{id}_{A}\mathbin{\otimes}\tau \mathbin{\otimes}\mathup{id}_{M}}
	A\mathbin{\otimes}A\mathbin{\otimes}M
	\xrightarrow{m_{A}\mathbin{\otimes}\mathup{id}_{M}}
	A\mathbin{\otimes}M
	.
\]
Similarly, if \( N \in \categoryformat{$A $-dgmod} \), then
we denote by~\( C \mathbin{\otimes_{R}^{\tau }} N \in \categoryformat{$C $-dgcomod} \)
the comodule whose underlying graded comodule is~\( C \mathbin{\otimes_{R}} N \), but whose differential is given by
\[
	d_{C \mathbin{\otimes_{R}^{\tau }} N}
	=
	d_{C \mathbin{\otimes_{R}} N}
	+
	d_{C \mathbin{\otimes_{R}} N}^{\tau }
\]
where
\[
	d_{C \mathbin{\otimes_{R}} M}^{\tau }
	\colon
	C \mathbin{\otimes_{R}} N
	\xrightarrow{ \Delta  \mathbin{\otimes_{R}} \mathup{id}_{N} }
	C \mathbin{\otimes_{R}} C \mathbin{\otimes_{R}} N
	\xrightarrow{ \mathup{id}_{C} \mathbin{\otimes_{R}} \tau  \mathbin{\otimes_{R}} \mathup{id}_{N} }
	C \mathbin{\otimes_{R}} A \mathbin{\otimes_{R}} N
	\xrightarrow{ \mathup{id}_{C} \mathbin{\otimes_{R}} \mathup{ac} }
	C \mathbin{\otimes_{R}} N
	.
\]
Finally, if \( M \in \categoryformat{$C $-dgcomod} \)
and~\( N \in \categoryformat{$A $-dgmod} \),
then we denote by~\( \operatorname{Hom}_{R}^{\tau } (M ,N ) \)
the \( k \)-complex whose underlying
graded \( k \)-module is~\( \operatorname{Hom}_{R}^{\smallbullet} (M ,N ) \)
and whose differential is
\[
	d_{\operatorname{Hom}_{R}^{\tau } (M ,N )}
	=
	d_{\operatorname{Hom}_{R}^{\smallbullet} (M ,N )}
	+
	d_{\operatorname{Hom}_{R}^{\smallbullet} (M ,N )}^{\tau }
\]
where \( d_{\operatorname{Hom}_{R}^{\smallbullet} (M ,N )}^{\tau } \)
takes~\( g \in \operatorname{Hom}_{R}^{\smallbullet} (M ,N ) \)
to
\[
	d_{\operatorname{Hom}_{R}^{\smallbullet} (M ,N )}^{\tau } (g )
	\colon
	M
	\xrightarrow{\mathup{ca}}
	C \mathbin{\otimes_{R}} M
	\xrightarrow{\tau \mathbin{\otimes}g}
	A \mathbin{\otimes_{R}} N
	\xrightarrow{\mathup{ac}}
	N
	.
\]
One easily checks using the condition on~\( \tau  \) that the obvious maps are
in fact isomorphisms of complexes
\[
	\operatorname{Hom}_{A}^{\smallbullet} (A \mathbin{\otimes_{R}^{\tau }} M ,N )
	\cong
	\operatorname{Hom}_{R}^{\tau } (M ,N )
	\cong
	\operatorname{Hom}_{C}^{\smallbullet} (M ,C \mathbin{\otimes_{R}^{\tau }} N ).
\]
In particular,
\[
	A \mathbin{\otimes_{R}^{\tau }} {\slot}
	\colon
	\categoryformat{$C $-dgcomod} \longto \categoryformat{$A $-dgmod}
\]
is
left adjoint to
\[
	C \mathbin{\otimes_{R}^{\tau }} {\slot}
	\colon
	\categoryformat{$A $-dgmod} \longto \categoryformat{$C $-dgcomod}
	.
\]
One may show that for~\( C = \operatorname{Bar} (A ) \),
the natural map~\( \tau  \colon C \to A \) that kills everything
except the~\smash{\( A^{\otimes 1}[1] \)}-component is a twisting cochain.
If \( A \)~is augmented, we also get a twisting cochain on~\( C = \operatorname{Bar}^{+} (A ) \) by map~\( \tau  \colon C \to A \) that takes~\smash{\( \overline{A}^{\otimes 1}[1] \)} into~\( A \).
Analogous results hold for~\( A \) being the augmented or non-augmented cobar construction of a coalgebra~\( C \).

The following proposition is well-known and classical over a field (and probably also over a ring, to the right people, but we found no reference):

\begin{proposition}\label{res:a_infty_qiso=homotopy_equivalence}
	If the dg-algebra~\( A \)
	is quasi-free as an \( R \)-dg-module,
	and if \( M,N \in \categoryformat{$A $-mod}_{\infty} (\categoryformat{QF} (R ) ) \),
	then an \ainfty-quasi-isomorphism~\( M \to N \)
	is a homotopy equivalence.
\end{proposition}

This shows that the category~\(
	\categoryformat{$A $-mod}_{\infty} (\categoryformat{QF} (R ) )
\)
is already derived.

\begin{proof}
	Let~\(
		f
		\colon
		M
		\to
		N
	\)
	be an \ainfty-quasi-isomorphism. Then
	\( f_{1} \colon M \to N \)~is a quasi-isomorphism
	of h-projective \( R \)-modules
	and hence a homotopy equivalence.
	Consider the map of \( C \)-comodules~\( \smash{
		f
		\colon
		C \mathbin{\otimes_{R}^{\tau }} M
		\to
		C \mathbin{\otimes_{R}^{\tau }} N
	} \).
	We filter the coalgebra~\( C = \operatorname{Bar}_{R} (A ) \)
	as an \( R \)-dg-module
	by letting~\(
		F^{n} C
		=
		\bigoplus_{i\le n} A^{\otimes i} [i]
	\).
	This allows us to also
	filter \( C \mathbin{\otimes_{R}^{\tau }} M \)
	and~\( C \mathbin{\otimes_{R}^{\tau }} N \)
	as \( R \)-dg-modules by
	letting \(\smash{
		F^{n} (C \mathbin{\otimes_{R}^{\tau }} M )
		=
		(F^{n} C) \mathbin{\otimes_{R}^{\tau }} M
	}\)
	and~\(
		F^{n} (C \mathbin{\otimes_{R}^{\tau }} N )
		=
		(F^{n} C) \mathbin{\otimes_{R}^{\tau }} N
	\).
	In both cases, taking associated graded kills the
	differential~\( d_{\tau } \).
	Since \( A \)~is h-projective,
	\(
		\mathup{gr}^{n} \varphi 
		\colon
		A^{\otimes n} [n] \mathbin{\otimes_{R}} M
		\to
		A^{\otimes n} [n] \mathbin{\otimes_{R}} N
	\)
	is a homotopy equivalence of \( R \)-dg-modules.
	Thus taking the cone
	\[\begin{tikzcd}
		C \mathbin{\otimes_{R}^{\tau }} M
		\ar[r,"f"]
			& C \mathbin{\otimes_{R}^{\tau }} N
			\ar[r]
				& C(f )
				\invcomma
	\end{tikzcd}\]
	we obtain a filtration on~\( C(f ) \)
	such that each graded piece is contractible over~\( R \).
	The exact sequence of dg-comodules
	\[
		C \mathbin{\otimes_{R}^{\tau }} N
		\longto
		C(f)
		\longto
		C \mathbin{\otimes_{R}^{\tau }} M[{1}]
	\]
	splits as a short exact sequence of \emph{graded}~\( C \)-comodules.
	Therefore, taking the graded hom, we get a short exact
	sequence of complexes of vector spaces
	\[
		\operatorname{Hom}_{C}^{\smallbullet} (C \mathbin{\otimes_{R}^{\tau }} M[{1}],C \mathbin{\otimes_{R}^{\tau }} M )
		\to
		\operatorname{Hom}_{C}^{\smallbullet} (C(f ),C \mathbin{\otimes_{R}^{\tau }} M )
		\to
		\operatorname{Hom}_{C}^{\smallbullet} (C \mathbin{\otimes_{R}^{\tau }} N ,C \mathbin{\otimes_{R}^{\tau }} M )
		.
	\]
	This therefore leads to a long exact sequence of the cohomologies.
	If we can prove that the middle term~\(
		\operatorname{Hom}_{C}^{\smallbullet} (C(f ),C \mathbin{\otimes_{R}^{\tau }} M )
	\)
	is exact, we will get that precomposition
	\[
		f^{*}
		\colon
		\operatorname{Hom}_{C}^{\smallbullet} (C \mathbin{\otimes_{R}^{\tau }} N ,C \mathbin{\otimes_{R}^{\tau }} M )
		\longto
		\operatorname{Hom}_{C}^{\smallbullet} (C \mathbin{\otimes_{R}^{\tau }} M ,C \mathbin{\otimes_{R}^{\tau }} M )		
	\]
	is a quasi-isormophism.
	The class~\smash{\(
		\mathopen[\mathup{id} \mathclose]
		\in
		H^0 \operatorname{Hom}_{C}^{\smallbullet} (C \mathbin{\otimes_{R}^{\tau }} M ,C \mathbin{\otimes_{R}^{\tau }} M )
	\)}
	will then determine an element
	in~\smash{\(
		H^0 \operatorname{Hom}_{C}^{\smallbullet} (C \mathbin{\otimes_{R}^{\tau }} N ,C \mathbin{\otimes_{R}^{\tau }} M )
	\)}
	which will be a homotopy inverse to~\( f \).
	
	In proving the claim that
	\(
		\operatorname{Hom}_{C}^{\smallbullet} (C(f ),C \mathbin{\otimes_{R}^{\tau }} M )
		=
		\operatorname{Hom}_{R}^{\tau } (C(f ),M )
	\)~is exact,
	we use the previously mentioned
	exhaustive \( C \)-dg-comodule filtration~\( F^{n} C(f ) \)
	for which the graded
	pieces~\( \mathup{gr}^{n} C(f ) \)
	are contractible over~\( R \),~\( F^{-1} C(f ) = 0 \),
	and where each filtered and graded piece is projective as a graded \( R \)-module. Therefore, the exact sequence of \( C \)-comodules
	\[
		0
		\to
		F^{n-1} C(f )
		\to
		F^{n} C(f )
		\to
		\mathup{gr}^{n} C(f )
		\to
		0
	\]
	splits as an exact sequence of graded \( R \)-modules.
	Since homs commute with direct sums, this yields an exact
	sequence
	\[
		0
		\to
		\operatorname{Hom}_{R}^{\tau } (\mathup{gr}^{n} C(f ),M )
		\to
		\operatorname{Hom}_{R}^{\tau } (F^{n} C(f ),M )
		\to
		\operatorname{Hom}_{R}^{\tau } (F^{n-1} C(f ),M )
		\to
		0.
	\]
	Taking the long exact sequence of cohomologies, we get
	by induction that
	\( \operatorname{Hom}_{R}^{\tau } (F^{n} C(f ),M ) \)~is exact for all~\( n \).
	Now we obtain the~\( C(f ) \)
	as the right term in the exact sequence
	\[
		0
		\longto
		\bigoplus_{n\ge 0} F^{n} C(f )
		\xrightarrow{\mathup{id}-s}
		\bigoplus_{n\ge 0} F^{n} C(f )
		\longto
		C(f )
		\longto
		0
	\]
	where \( s \)~is the sum of the inclusions~\(
		F^{n} C(f )
		\into
		F^{n+1} C(f )
	\).
	As before, this sequence consists of complexes
	which are projective as graded modules,
	so the sequence splits as a short exact sequence of graded modules.
	As above, we get that~\( \operatorname{Hom}_{R}^{\tau } (C(f ),M ) \)~is exact, as claimed.
	
	We provide an alternative argument using Mittag--Leffler:
	As before, we argue that the
	complex~\(
		\operatorname{Hom}_{R}^{\tau } (F^{n} C(f ),M )
	\)
	is exact for all~\( n \).
	Now since the maps~\(
		\operatorname{Hom}_{R}^{\tau } (F^{n} C(f ),M )
		\to
		\operatorname{Hom}_{R}^{\tau } (F^{n-1} C(f ),M )
	\)
	are surjective, the inverse system~\(
		\operatorname{Hom}_{R}^{\tau } (F^{n} C(f ),M )
	\)
	satisfies Mittag--Leffler. Therefore,
	\[
		\operatorname{Hom}_{R}^{\tau } (C(f ),M )
		=
		\operatorname{Hom}_{R}^{\tau } (\dirlimformat{lim} F^{n} C(f ),M )
		=
		\invlimformat{lim} \operatorname{Hom}_{R}^{\tau } (F^{n} C(f ),M )
	\]
	is exact (both limits being unenriched limits evaluated in the category of graded modules).
\end{proof}

\begin{lemma}\label{lemma:CobBar_resolution}
	If \( A \)~is a dg-algebra over~\( R \)
	and~\( M \in \categoryformat{$A $-dgmod} \),
	then the counit of adjunction
	\[
		A
		\mathbin{\otimes_{R}^{\tau }}
		\operatorname{Bar} (A )
		\mathbin{\otimes_{R}^{\tau }}
		M
		\to
		M
	\]
	is a quasi-isomorphism of \( A \)-dg-modules.
\end{lemma}

\begin{proof}
	This is exactly the bar resolution of~\( M \).
\end{proof}

\begin{lemma}\label{res:ainfty_bar_resolution}
	If \( A \)~is an augmented, flat dg-algebra over dg-algebra~\( R \) over~\( k \),
	and \( M \)~is a strictly unital \ainfty-module over~\( A \) which is flat over~\( R \),
	then the unit of adjunction~\(
		M
		\to
		A \mathbin{\otimes_{R}^{\tau }} \operatorname{Bar}^{+} (M )
	\) is a quasi-isormophism of \ainfty-modules.
\end{lemma}

\begin{proof}
	We claim in fact that
	\( \smash{
		(M ,\mathup{ac}_{1} )
		\to
		(A \mathbin{\otimes_{R}^{\tau }} \operatorname{Bar}^{+} (M ),\mathup{ac}_{1} )
	} \)~is a filtered quasi-isomorphism.
	On both sides, filter~\( M \) by the trivial filtration~\( F^{i} M = M \) for all~\( i\ge 0 \).
	Filter~\smash{\( C = \operatorname{Bar}^{+} (A ) = \bigoplus_{n} \smash{\overline{A}}^{\otimes n} [n] \)}
	by~\smash{\( F^{i} C = \bigoplus_{n\le i} \smash{\overline{A}}^{\otimes n} [n] \)}
	and~\( A \)
	by~\( F^{0} A = R \)
	and~\( F^{i} A = A \)
	for~\( i > 0 \).
	Since all the tensor factors are flat, this induces
	a filtration on~\(
		A
		\mathbin{\otimes_{R}^{\tau }}
		%\vC \tensor[twist=\vtau,\vR] \vM
		\operatorname{Bar}^{+} (M )
	\)
	with~\(
		\mathup{gr}^{0} (A \mathbin{\otimes_{R}^{\tau }} \operatorname{Bar}^{+} (M ) )
		=
		M
	\)
	and
	\[
		\mathup{gr}^{i} (A \mathbin{\otimes_{R}^{\tau }} \operatorname{Bar}^{+} (M ) )
		=
		(\overline{A} \mathbin{\otimes_{R}} \mathup{gr}^{i-1} C \mathbin{\otimes_{R}} M )
		\oplus
		(\mathup{gr}^{i} C \mathbin{\otimes_{R}} M )
		\qquad\text{for~\( i>0 \)}
		.
	\]
	We claim that \( \mathup{gr}^{i} (A \mathbin{\otimes_{R}^{\tau }} \operatorname{Bar}^{+} (M ) ) \)~is contractible for~\( i>0 \).
	Indeed, the differential is given by
	\[
		d_{\mathup{gr}^{i} (A \mathbin{\otimes_{R}^{\tau }} \operatorname{Bar}^{+} (M ) )}
		=
		\begin{pmatrix}
			d_{\overline{A} \mathbin{\otimes_{R}} \mathup{gr}^{i-1} C \mathbin{\otimes_{R}} M}
				& p
		\\
			0
				& d_{\mathup{gr}^{i} C \mathbin{\otimes_{R}} M}
		\end{pmatrix}
	\]
	where \(
		p
		\colon
		\mathup{gr}^{i} C \mathbin{\otimes_{R}} M
		\to
		\smash{\overline{A}} \mathbin{\otimes_{R}} \mathup{gr}^{i-1} C \mathbin{\otimes_{R}} M
	\)
	comes from the isomorphism~\(
		\mathup{gr}^{i} C
		\cong
		\smash{\overline{A}} [1] \mathbin{\otimes_{R}} \mathup{gr}^{i-1} C
	\).
	Clearly, letting
	\[
		s
		=
		\begin{pmatrix}
			0 & 0
		\\
			\mathup{id}&0
		\end{pmatrix},
	\]
	we get~\( \mathup{id} = d s + s d \).
\end{proof}

\begin{lemma}\label{lemma:CobBar_H-Proj}
	Let~\( A \) be a dg-algebra over a dg-algebra~\( R \)
	over~\( k \)
	and let~\( M \in \categoryformat{QF} (A ) \).
	If~\( A \in \categoryformat{QF} (R ) \),
	then~\(
		A
		\mathbin{\otimes_{R}^{\tau }}
		\operatorname{Bar} (A )
		\mathbin{\otimes_{R}^{\tau }}
		M
		\in
		\categoryformat{QF} (A )	
	\).
	Similarly, if \( A \)~is augmented
	and \smash{\( \overline{A} \in \categoryformat{QF} (R ) \)},
	then~\(
		A
		\mathbin{\otimes_{R}^{\tau }}
		\operatorname{Bar}^{+} (A )
		\mathbin{\otimes_{R}^{\tau }}
		M
		\in
		\categoryformat{QF} (A )
	\).
\end{lemma}

\begin{proof}
	Filter the bar constructions by the number of tensor factors to kill the bar differential and the \( \tau  \)-differential.
	Then apply
	\cref{res:tensors_of_quasi-free,res:scalar_extension_of_quasi-free}.
\end{proof}

\begin{proposition}\label{res:der_as_A_infty_modules}
	Let \( A \) be an augmented dg-algebra over a dg-algebra~\( R \) over~\( k \) such that \( A \)~is quasi-free as an \( R \)-dg-module.
	Then we have a quasi-equivalence of dg-categories
	\[
		\mathscr{D} (\categoryformat{$A $-dgmod} )
		\cong
		\categoryformat{$A $-mod}_{\infty} (\categoryformat{QF} (R ) )
		.
	\]
\end{proposition}

\begin{proof}
	We use the presentation~\(
		\mathscr{D} (\categoryformat{$A $-dgmod} )
		\cong
		\categoryformat{QF} (A )
	\).
	Since \( M \)~is quasi-free over~\( A \)
	and \( A \)~is quasi-free over~\( R \),
	we obtain from \cref{lemma:restriction_preserves_H-Proj} that~\( M|_{R} \in \categoryformat{QF} (R ) \).
	Thus we
	obtain a dg-functor
	\[
		\categoryformat{QF} (A )
		\longto
		\categoryformat{$A $-mod}_{\infty} (\categoryformat{QF} (R ) ).
	\]
	To see that it is fully faithful,
	we notice that if \(
		M , N
		\in
		\categoryformat{$A $-mod}_{\infty} (\categoryformat{QF} (R ) )
	\),
	then
	\begin{align*}
		\MoveEqLeft
		\operatorname{Hom}_{\categoryformat{$A $-mod}_{\infty} (\categoryformat{QF} (R ) )} (M ,N )
			=
			\operatorname{Hom}_{\operatorname{Bar}^{+} (A )} (\operatorname{Bar}^{+} (A ) \mathbin{\otimes_{R}^{\tau }} M ,\operatorname{Bar}^{+} (A ) \mathbin{\otimes_{R}^{\tau }} N )
	\\
			&=
			\operatorname{Hom}_{A} (A \mathbin{\otimes_{R}^{\tau }} \operatorname{Bar}^{+} (A ) \mathbin{\otimes_{R}^{\tau }} M ,N )
				\cong
				\operatorname{Hom}_{\mathscr{D} (\categoryformat{$A $-dgmod} )} (M ,N )
	\end{align*}
	where the last equivalence
	is by~\cref{res:ainfty_bar_resolution,lemma:CobBar_H-Proj}.
	Quasi-essential surjectivity also follows from \cref{res:ainfty_bar_resolution}.
\end{proof}

\begin{proposition}\label{res:A_infty_hu_is_A_infty}
	For any augmented dg-algebra~\( A \)
	which is quasi-free over~\( R \),
	the inclusion~\(
		\categoryformat{$A $-mod}_{\infty} (\categoryformat{QF} (R ) )
		\subset
		\categoryformat{$A $-mod}_{\infty}^{\categoryformat{hu}} (\categoryformat{QF} (R ) )
	\)
	is a quasi-equivalence of dg-categories.
\end{proposition}

We believe this also holds for non-augmented dg-algebras,
following a proof like~\textcite{barcategory}, but the proof is simpler in this special case.

\begin{proof}
	We must prove that
	the embedding is quasi-essentially surjective,
	i.e.\ that any homotopy-unital module can
	be resolved by a strictly unital one.
	If~\( M \in \categoryformat{$A $-mod}_{\infty}^{\categoryformat{hu}} (\categoryformat{QF} (R ) ) \),
	we may regard~\( M \) as a strictly unital \( A \oplus R \)-module. Applying \cref{res:ainfty_bar_resolution},
	we obtain a quasi-isomorphism~\smash{\(
		M
		\to
		(A \oplus R )
		\mathbin{\otimes_{R}^{\tau }}
		\operatorname{Bar}_{A \oplus R}^{+} (M )
	\)}
	over~\( A \oplus R \).
	Restricting by the unital dg-algebra map~\smash{\( A = \overline{A} \oplus R \into A \oplus R \)},
	we obtain that it is also a quasi-equivalence over~\( A \).
	Since 
	\(
		(A \oplus R )
		\mathbin{\otimes_{R}^{\tau }}
		\operatorname{Bar}_{A \oplus R}^{+} (M )
	\)~is strictly unital over~\( A \), we have proved essential surjectivity.
	It is also quasi-free over~\( R \), as we can filter the bar construction to kill the bar differential until we are left with a tensor product of quasi-free \( R \)-modules.
	To prove quasi-fully faithfulness, we note that the calculation of the hom space in~\( \mathscr{D} (\categoryformat{$A $-dgmod} ) \)
	in the previous proof works no matter if we are using the augmented or non-augmented bar construction
	(one then applies \cref{lemma:CobBar_resolution} in place of~\cref{res:ainfty_bar_resolution}).
	Therefore, hom spaces agree on~\( H^{0} \).
	By shifting, they also agree on~\( H^{i} \) for all~\( i \).
\end{proof}

\endgroup

\begingroup

\chapter{Koszul duality}

The material below is a cross between three approaches to Koszul duality
for triangulated categories of modules. Firstly, there is a general approach of \textcite{positselski}: the author considers a pair of a dg-algebra~\( A \) over a ring~\( \mathbb{K} \) and
a dg-coalgebra~\( C \) given by its bar construction. The duality couples certain
\emph{exotic} derived categories of dg-modules (resp.,~dg-comodules).

Secondly, there is a similar approach due to \textcite{keller} with one difference: we
stay in the world of usual derived categories all the way. Keller's Koszul
duality is a special case of a general theorem describing a subcategory in a
dg-category~\( \mathscr{C} \) generated by a compact object \( M \in \mathscr{C} \).

Lastly, there is an approach of \textcite{linear_koszul}.
Its advantage is that the authors work over a possibly non-affine base.
Its drawback is that the authors work in categories of \emph{graded} dg-modules over
dg-algebras equipped with the second (inner) grading. We mostly retell
linear Koszul duality below replacing the use of the inner grading by Keller's
considerations.

\section{Koszul complex}

In what follows we develop a story in the spirit of
quadratic-linear-scalar duality due to Positselski.

Recall the most classical setting for it. Let~\( \mathbb{K} \) be a field of characteristic zero, and fix a commutative \( \mathbb{K} \)-algeba~\( S \) and
a finite rank free complex of \( S \)-modules~\( M \).
Denote~\( M^{*} = \operatorname{Hom}_{S} (M ,S ) \).
Consider the pair of quadratic dual dg-algebras \( A = \operatorname{Sym}_{S} (M[1] ) \)
and~\( A^{!} = \operatorname{Sym}_{S} (M^{*} [{-2}] ) \), with differentials obtained from the one in~\( M \) by the Leibniz rule. Notice that symmetric algebra in the definition is understood in the
graded sense.
Letting~\( \tau  \colon A^{*} \to A^{!} [{1}] \) be the twisting
cochain given by the identity on~\( M^{*} \),
it is known that
\(
	K
	=
	\smash{(A^{*} \mathbin{\otimes_{S}^{\tau }} A^{!},d_{A^{*} \mathbin{\otimes_{S}^{\tau }} A^{!}} )}
\)~has a structure of an
\( A \)--\( A^{!} \)-dg-bimodule quasi-isomorphic to the trivial module~\( S \) both as an \( A \)- and as an \( A^{!} \)-module. The bimodule~\( K \) is called the Koszul complex.

\begin{remark}
	One way to find cohomology of~\( K \) is as follows. Consider the complex of \( S \)-modules~\( C_{M} = C(\mathup{id}_{M} )[{1}] \). The Koszul complex~\( K \) can be realized as~\( \operatorname{Sym}_{S} (C_{M} ) \) as a graded \( S \)-module. Evidently, it is quasi-isomorphic to~\( S \). The commuting actions of \( A \) and~\( A^{!} \) are recovered in the following way. Consider the complex of \( S \)-modules~\( C_{M}\oplus C_{M}^{*} \) equipped with the canonical non-degenerate, graded skew-symmetric bilinear form. This allows us to define the \textdef{Heisenberg algebra}~\(
		\operatorname{Heis}_{S} (C_{M} \oplus{C_{M}^{*}} )
	\)
	associated to~\( C_{M} \oplus{C_{M}^{*}} \)
	as the free algebra with generators~\( C_{M} \oplus{C_{M}^{*}} \)
	and relations~\(
		a b - (-1)^{\lvert a\rvert \lvert b\rvert} b a
		=
		\langle a,b \rangle
	\).
	It can be rewritten as
 	\[
 		\operatorname{Heis}_{S} (C_{M} \oplus{C_{M}^{*}} )
 		=
 		\operatorname{Heis}_{S} (M[2] \oplus M^{*}[-1] )
 		\mathbin{\otimes_{S}}
 		\operatorname{Heis}_{S} (M[{1}] \oplus M^{*} [-2] )
 	\]
	and maps naturally into the endomorphism algebra of the \( S \)-module~\( K \). It contains the commuting  subalgebras \( A = \operatorname{Sym}_{S} (M[1] ) \)
	and~\( A^{!} = \operatorname{Sym}_{S} (M^{*} [-2] ) \). One checks directly that
	\(
		\smash{\operatorname{Heis}_{S} (C_{M} \oplus{C_{M}^{*}} )}
	\)~is a dg-subalgebra in~\( \operatorname{End}_{S} (K ) \) and both subalgebras \( A \) and~\( A^{!} \) are dg-subalgebras
	in~\(
		\smash{\operatorname{Heis}_{S} (C_{M} \oplus{C_{M}^{*}} )}
	\).
\end{remark}

\section{Twisted case}\label{sec:twisted_case}

Next we replace the ring~\( S \) with a graded algebra and make its interaction with \( A \) and~\( A^{!} \) more complicated. Fix a commutative ring~\( \mathbb{K} \) of characteristic zero.
Take  two free \( \mathbb{K} \)-modules of ﬁnite rank \( M \) and~\( N \).
Let \( \varphi  \colon M \to N^{*} \) be a \( \mathbb{K} \)-linear map. We  consider the dg-algebras over \( \mathbb{K} \) given by 
\begin{align*}
	A &= \operatorname{Heis}_{\mathbb{K}} (M[{1}] \oplus N[{-1}] ),
\\
	A^{!} &= \operatorname{Sym}_{\mathbb{K}} (N[{-1}] \xto{\varphi^{*}}{M^{*} [{-2}]} ).
\end{align*}
The first algebra is the Heisenberg algebra on~\( M[{1}]\oplus N[{-1}] \)
with respect to the canonical, graded skew-symmetric bilinear form induced by~\( \varphi  \). 

\begin{remark}
	Clearly this form restricts trivially both to \( M[{1}] \) and to~\( N[{-1}] \). It follows that \( S = \operatorname{Sym}_{\mathbb{K}} (N[{-1}] ) \)~is a subalgebra in both \( A \) and~\( A^{!} \). However, \( S \)~is not a dg-subalgebra in~\( A^{!} \). Also, \( A \)~is not a free graded skew-commutative algebra over~\( S \) (in fact, \( S \)~is not even central in~\( A \)). In this sense, the previous case is a toy example for the present one.
\end{remark}

We generalize the construction of the Koszul complex. Consider the complex of \( \mathbb{K} \)-modules:
\[
	L
	=
	\bigl((M^{*} \oplus N )[{-1}] \xrightarrow{(-\mathup{id}_{M^{*}},\varphi^{*} )} M^{*} [{-2}] \bigr).
\]
Clearly it is quasi-isomoprhic to~\( N[{-1}] \). 
Consider~\( K = \operatorname{Sym}_{\mathbb{K}} (L ) \) with the differential extended to symmetric powers by the Leibniz rule.
It is quasi-isomorphic to~\( \operatorname{Sym}_{\mathbb{K}} (N[{-1}] ) \) as a complex of \( \mathbb{K} \)-modules.
Generalizing the toy example, we construct two commuting actions of the dg-algebras \( A \) and~\( A^{!} \) on~\( K \) explicitly.

Consider the Heisenberg algebra of the complex~\( L\oplus L^{*} \),
where we write~\( L^{*} = \operatorname{Hom}_{S} (L ,S ) \):
\[
	\operatorname{Heis}_{\mathbb{K}} (L \oplus L^{*} )
	=
	\operatorname{Heis}_{\mathbb{K}} \bigl(M[{2}] \oplus N[{1}] \oplus M[{1}] \oplus{M^{*} [{-1}]} \oplus N[{-1}] \oplus{M^{*} [{-2}]} \bigr).
\]
It embeds into the dg-algebra of endomorphisms of~\( \operatorname{Sym}_{\mathbb{K}} (L ) \) and acts naturally on~\( K \).

Notice that the graded algebra~\(
	A = \operatorname{Heis}_{\mathbb{K}} (M[{1}] \oplus N[{-1}] )
\)
embeds into the algebra~\(
	\operatorname{Heis}_{\mathbb{K}} (L \oplus L^{*} )
\)
in the following way: While the embedding of~\( M[{1}] \) is the obvious one,
\(
	N[{-1}]
\)~maps
to~\(
	M^{*} [{-1}]\oplus N[{-1}]
\)
via~\(
	(\varphi^{*},\mathup{id} )
\).
One checks directly that both the relations and the differentials match.

Notice  that the dg-algebra~\(
	A^{!}
	=
	\operatorname{Sym}_{\mathbb{K}} (N[{-1}] \to{M^{*} [{-2}]} )
\)
also embeds naturally into \( \operatorname{Heis}_{\mathbb{K}} (L \oplus L^{*} ) \).
Moreover its image belongs to the centralizer of the image of~\( A \).
We proved the following statement.

\begin{lemma}\label{res:K_is_a_resolution}
	The complex \( K \) has a natural structure of an
	\( A \)--\( A^{!} \)-dg-bimodule quasi-isomorphic
	to~\(
		S = \operatorname{Sym}_{\mathbb{K}} (N[{-1}] )
	\)
	both as an \( A \)-module and as an \( A^{!} \)-module. 
\end{lemma}

\begin{remark}\label{res:K_is_quasi-free}
	In fact, it is not hard to see that~\(
		K \cong A^{!} \mathbin{\otimes^{\tau }} A^{*}
	\),
	where the twisting cochain~\( \tau  \colon A^{*} \to A^{!} [{1}] \) is given by the identity on~\( M^{*} [{-1}] \).
	Note that by filtering~\( A^{*} \) by the number of tensor
	factors, we kill the twisted differential, and hence we obtain a filtration of~\( K \) by a finite number of free, finitely generated \( A^{!} \)-modules. Therefore, \( K \)~is semifree over~\( S \) (and hence quasi-free). Also, \( K \)~is free of finite sank as a graded \( A^{!} \)-module and
	therefore compact and h-projective in \( \categoryformat{$A^{!}$-dgmod} \).
\end{remark}

\section{Compact generators}

Recall that an object~\( M \) in a pretriangulated dg-category~\( \mathscr{C} \)
is called \textdef{compact} 
if the functor~\(
	\operatorname{Hom}_{\mathscr{C}} (M ,\slot )
\)
commutes with direct sums.
We call~\( M \) a \textdef{compact generator} if its
\textdef{right orthogonal}~\( M^{\perp} \) in the homotopy category,
i.e.~the full subcategory of objects annihilated
by~\(
	H^{0} (\operatorname{Hom}_{\mathscr{C}} (M ,\slot ) )
\),
is~\( 0 \).

\begin{theorembreak}[Theorem {\parencite[section 7.3]{keller}}]
	Consider a dg-algebra~\( S \) over a commutative ring~\( \mathbb{K} \).
	Let~\( A \) be a dg-algebra in the category of dg-bimodules over~\( S \).
	For~\( M\in \categoryformat{$A $-dgmod} \) a compact generator,
	 we have a quasi-equivalence of dg-categories
	 \[
	 	\mathscr{D} (\categoryformat{$A $-dgmod} )
	 	\longisoto
	 	\mathscr{D} (\categoryformat{$E $-dgmod} )
	 \]
	 given by~\(
	 	N \mapsto \mathup{R} \!\operatorname{Hom}_{A} (M ,N )
	 \).
	 Here, \( E \) denotes the dg-algebra~\(
	 	E
	 	=
	 	\mathup{R} \!\operatorname{Hom}_{A} (M ,M )^{\mathup{op}}
	 \).
\end{theorembreak}

The generator condition can be omitted in the following way.
Slightly abusing notation we denote the full subcategory
in~\( \mathscr{D} (\categoryformat{$A $-dgmod} ) \) consisting of objects whose images belong to the right orthogonal to~\( M\in \categoryformat{$A $-dgmod} \)
by~\( M^{\perp} \).
Take the Drinfeld quotient dg-category~\(
	\mathscr{D} (\categoryformat{$A $-dgmod} )/{M^{\perp}}
\).

\begin{corollary}\label{res:pre_Koszul_duality}
	For a compact object~\( M\in \categoryformat{$A $-dgmod} \)  we have a quasi-equivalence of dg-categories
	\[
		\mathscr{D} (\categoryformat{$A $-dgmod} )/{M^{\perp}}
		\longisoto
		\mathscr{D} (\categoryformat{$E $-dgmod} )
	\]
	given by the same functor as in the previous theorem. 
\end{corollary}

\begin{remark}\label{rem:recollement}
	It is convenient to identify the quotient dg-category
	\[
		\mathscr{D} (\categoryformat{$A $-dgmod} )/{M^{\perp}}
	\]
	with the subcategory in~\( \mathscr{D} (\categoryformat{$A $-dgmod} ) \)
	generated by~\( M \), i.e.~with the minimal
	pretriangulated subcategory in~\( \mathscr{D} (\categoryformat{$A $-dgmod} ) \)
	containing~\( M \)
	that has arbitrary small coproducts.
	Denote the latter by \( \langle M \rangle  \).
	The previous corollary means that
	\[
		\langle M \rangle 
		\longisoto
		\mathscr{D} (\categoryformat{$E $-dgmod} )
	\]
	for~\(
		E=\mathup{R} \!\operatorname{Hom}_{A} (M ,M )^{\mathup{op}}
	\).
\end{remark}

\section{Geometric setting}

We pass to our case of interest. Let~\( X \) be a regular Noetherian affine scheme over a field of characeristic zero, and let~\( \mathbb{K} = \mathscr{O}_{X} \) be its ring of regular functions.
Consider two finite rank vector bundles \( M \) and~\( N \)
on~\( X \) and a vector bundle map~\( \varphi  \colon M \to N^{*} \)
like in~\cref{sec:twisted_case}.
We obtain a pair of dg-algebras
\( A^{!} \) and~\( A \) and their
Koszul complex~\( K \)
which is a resolution of~\( S \).
In other words,
\begin{align*}
	A
		&= \operatorname{Heis}_{\mathscr{O}_{X}} (M[{1}] \oplus N[{-1}] ),
\\
	A^{!}
		&= \operatorname{Sym}_{\mathscr{O}_{X}} (N[{-1}] \to{M^{*} [{-2}]} ),
\\
	S
		&= \operatorname{Sym}_{\mathscr{O}_{X}} (N[{-1}] ).
\end{align*}
We work in the dg-category~\(
	\mathscr{C} = \categoryformat{$A^{!}$-dgmod}
\).
Recall from~\cref{res:K_is_quasi-free} that \( K \) is an h-projective and quasi-free resolution of~\( S \) in~\( \mathscr{C} \).

Denote the dg-algebra~\( \operatorname{Hom}_{A^{!}} (K ,K )^{\mathup{op}} \)
by~\( E \).

\begin{corollary}
	We have a quasi-equivalence of dg-categories
	\[
		\mathscr{D} (\categoryformat{$A^{!}$-dgmod} )/{S^{\perp}}
		\longisoto
		\mathscr{D} (\categoryformat{$E $-dgmod} ).
	\]
\end{corollary}

\begin{remark}
	Following the usual geometric intuition, we denote the
	triangualted subcategory~\( \langle S \rangle  \) by~\(
		\mathscr{D} (\categoryformat{$A^{!}$-dgmod} )_{\categoryformat{tors}}
	\).
	This way, using~\cref{rem:recollement},
	the quasi-equivalence from~\cref{res:pre_Koszul_duality} reads as follows:
	\[
		\mathscr{D} (\categoryformat{$A^{!}$-dgmod} )_{\categoryformat{tors}}
		\longisoto
		\mathscr{D} (\categoryformat{$E $-dgmod} ).
	\]
\end{remark}

\begin{lemma}
	We have a quasi-isomorphism of dg-algebras~\(
		A
		\longisoto
		\operatorname{Hom}_{A^{!}} (K ,K )^{\mathup{op}}
	\).
\end{lemma}

\begin{proof}
	We have a map
	\[
		A
		\longto
		\operatorname{Hom}_{S} (A^{*},A^{*} )^{\mathup{op}}
		\longto
		\operatorname{Hom}_{A^{!}} (A^{!} \mathbin{\otimes_{S}^{\tau }} A^{*},A^{!} \mathbin{\otimes_{S}^{\tau }} A^{*} )^{\mathup{op}}
	,\]
	which also shows that
	\(
		\operatorname{Hom}_{S} (K ,K )^{\mathup{op}}
	\)~is a dg-algebra over~\( S \).
	To see that
%	the map~\(
%		\vA \to \Hom[\vA[!dual],*]{\vK,\vK}
%	\)
	this map
	is a quasi-isomorphism, we observe that the fact that
	\( K \)~is a quasi-free resolution of~\( S \)
	(see~\cref{res:K_is_a_resolution,res:K_is_quasi-free})
	implies that
	\smash{\(
		\operatorname{Hom}_{A^{!}}^{\smallbullet} (K ,K )^{\mathup{op}}
	\)}~calculates~\smash{\(
		\mathup{R} \!\operatorname{Hom}_{A^{!}} (S ,S )^{\mathup{op}}
	\)}.
	Therefore, the augmentation~\(
		A^{!} \to S
	\)
	induces a quasi-isomorphism
	\[
		\operatorname{Hom}_{A^{!}}^{\smallbullet} (K ,K )^{\mathup{op}}
		\longisoto
		\operatorname{Hom}_{A^{!}}^{\smallbullet} (K ,S )
		=
		\operatorname{Hom}_{S}^{\tau } (A^{*},S )
		=
		A
		.
	\]
	One checks directly that
	the composition~\smash{\(
		A
		\to
		\operatorname{Hom}_{A^{!}}^{\smallbullet} (K ,K )^{\mathup{op}}
		\to
		A
	\)}
	is the identity.
	By the two-out-of-three property, the first map
	is a quasi-isomorphism.
	Therefore, as an algebra, \(
		\operatorname{Hom}_{A^{!}}^{\smallbullet} (K ,K )^{\mathup{op}}
	\)~is quasi-isomorphic to~\( A \).
\end{proof}

We proved the following statement:

\begin{proposition}\label{res:Koszul_duality}
	We have a quasi-equivalence of dg-categories
	\[
		\mathscr{D} (\categoryformat{$A^{!}$-dgmod} )_{\categoryformat{tors}}
		\longisoto
		\mathscr{D} (\categoryformat{$A $-dgmod} ) .
	\]
\end{proposition}

\endgroup

\chapter{Homotopy limits}

For a general introduction to the homotopy limit machinery we shall use, we refer 
the reader to \textcite{ends}.
To us, a homotopy limit will be always understood as the derived functor of the limit functor in the model category-theoretic sense, and it serves as a concrete realization of the \( \infty \)-categorical limit functor. In particular, the homotopy limit preserves weak equivalences.

\section{General lemmata on homotopy limits}

The following section adds a few technical lemmata to \textcite{ends} which we shall need.

\begin{lemma}\label{res:fubini_holim}
	Let~\( \mathscr{C} \) be a combinatorial model category
	and~\( \Gamma_{1}, \Gamma_{2} \) two categories.
	The homotopy limit
	\(
		\invlimformat{holim}_{\Gamma_{1} \times \Gamma_{2}} F 
	\)
	may be calculated componentwise, i.e.
	\[
		\invlimformat{holim}_{(\gamma_{1},\gamma_{2} ) \in \Gamma_{1} \times \Gamma_{2}} F(\gamma_{1},\gamma_{2} )
		=
		\invlimformat{holim}_{\gamma_{1} \in \Gamma_{1}} \bigl(\invlimformat{holim}_{\gamma_{2} \in \Gamma_{2}} F(\gamma_{1},\gamma_{2} ) \bigr)
	\]
	where both homotopy limits on the right-hand side are
	derived functors of the
	pointwise limit~\( \smash{
		\invlimformat{lim}_{\Gamma_{i}}
		\colon
		\mathscr{C}^{\Gamma_{i}}
		\to
		\mathscr{C}
	} \).
\end{lemma}

\begin{proof}
	Recall from \textcite{ends}
	that the limit functor
	\[
		\invlimformat{lim}_{\Gamma_{1} \times \Gamma_{2}}
		\colon
		\mathscr{C}_{\mathup{inj}}^{\Gamma_{1} \times \Gamma_{2}}
		\to
		\mathscr{C}
	\]
	is right Quillen,
	where~\(
		\smash{\mathscr{C}}_{\mathup{inj}}^{\Gamma_{1} \times \Gamma_{2}}
	\)
	is the category of diagrams~\(
		\Gamma_{1}\times\Gamma_{2}\to\mathscr{C}
	\)
	equipped with the injective model structure. Hence
	we have
	\[
		\invlimformat{holim}_{\Gamma_{1} \times \Gamma_{2}} F 
		=
		\invlimformat{lim}_{\Gamma_{1} \times \Gamma_{2}} R(F )
		=
		\invlimformat{lim}_{\gamma_{1} \in \Gamma_{1}} \bigl(\invlimformat{lim}_{\gamma_{2} \in \Gamma_{2}} R(F )(\gamma_{1},\gamma_{2} ) \bigr)
	\]
	where \( R(F ) \)~is a fibrant replacement
	of~\( F \) in~\( 
		\smash{\mathscr{C}}_{\mathup{inj}}^{\Gamma_{1} \times \Gamma_{2}}
	\).
	Now~\(
		\smash{\mathscr{C}}_{\mathup{inj}}^{\Gamma_{1} \times \Gamma_{2}}
		=
		\smash{\smash{(\mathscr{C}_{\mathup{inj}}^{\Gamma_{2}} )}}_{\mathup{inj}}^{\Gamma_{1}}
	\)
	(since they clearly have the same cofibrations and trivial cofibrations),
	so \( R(F )(\gamma_{1},\slot ) \)~is fibrant
	in~\( \smash{\mathscr{C}}_{\mathup{inj}}^{\Gamma_{2}} \)
	for all~\( \gamma_{1}\in\Gamma_{1} \).
	This shows that
	\(
		\smash{\invlimformat{lim}_{\gamma_{2} \in \Gamma_{2}} R(F )(\gamma_{1},\gamma_{2} )}
	\)
	calculates~\(
		\smash{\invlimformat{holim}_{\gamma_{2} \in \Gamma_{2}} F(\gamma_{1},\gamma_{2} )}
	\)
	for all~\( \gamma_{1}\in\Gamma_{1} \).
	Furthermore,
	the limit functor
	\[
		\invlimformat{lim}_{\Gamma_{2}}
		\colon
		\smash{\smash{(\mathscr{C}_{\mathup{inj}}^{\Gamma_{2}} )}}_{\mathup{inj}}^{\Gamma_{1}}
		\to
		\smash{\mathscr{C}}_{\mathup{inj}}^{\Gamma_{1}}
	\]
	is right Quillen by~\textcite[Remark A.2.8.6]{htt}.
	Thus~\(
		\smash{\invlimformat{lim}_{\gamma_{2} \in \Gamma_{2}} R(F )(\slot ,\gamma_{2} )}
	\)~is in fact
	a fibrant realization of
	the object~\(
		\smash{\invlimformat{holim}_{\gamma_{2} \in \Gamma_{2}} F(\slot ,\gamma_{2} )}
	\)
	in~\( \smash{\mathscr{C}}_{\mathup{inj}}^{\Gamma_{1}} \).
	Therefore,
	\[
		\invlimformat{lim}_{\gamma_{1} \in \Gamma_{1}} \bigl(\invlimformat{lim}_{\gamma_{2} \in \Gamma_{2}} R(F )(\gamma_{1},\gamma_{2} ) \bigr)
		=
		\invlimformat{holim}_{\gamma_{1} \in \Gamma_{1}} \bigl(\invlimformat{holim}_{\gamma_{2} \in \Gamma_{2}} F(\gamma_{1},\gamma_{2} ) \bigr)
	\]
	as claimed.
\end{proof}

\begin{lemma}\label{res:delta_sifted}
	The category~\( \mathbf{\Delta} \) is sifted, i.e.\ 
	the diagonal embedding~\( \mathbf{\Delta}\to\mathbf{\Delta}\times\mathbf{\Delta} \)
	is homotopy-initial. In other words,
	\[
		\invlimformat{holim}_{\mathopen[n\mathclose] \in \mathbf{\Delta}} F(\mathopen[n\mathclose],\mathopen[n\mathclose] )
		=
		\invlimformat{holim}_{(\mathopen[n\mathclose],\mathopen[m\mathclose] ) \in \mathbf{\Delta} \times \mathbf{\Delta}} F(\mathopen[n\mathclose],\mathopen[m\mathclose] ).
	\]
\end{lemma}

\begin{proof}
	The statement that~\( \mathbf{\Delta} \) is sifted is proved
	in~\textcite[Example~21.5]{dugger_homotopy_limits}.
	For a proof that homotopy-initial functors
	preserve homotopy limits,
	see e.g.~\textcite[Theorem~6.1]{ends}.
\end{proof}

\begin{example}\label{ex:semidirect_quotient}
	Let~\( G \)
	be a group scheme which is the semidirect product~\( G = N \rtimes H \)
	of a normal subgroup~\( N \) and a subgroup~\( H \),
	and let it act on a scheme~\( X \).
	We claim that there is an isomorphism of stack quotients
	\[
		\mathopen[\mathopen[X/{N} \mathclose]/{H} \mathclose]
		\cong
		\mathopen[X/{G} \mathclose]
		.
	\]
	To see this, write \(
		\varphi_{g} \colon N \to N
	\)
	for the action of~\( g \in G \) on~\( N \).
	Now notice that we have an isomorphism
	\[
		\mathopen[\mathopen[X/{N} \mathclose]_{n}/{H} \mathclose]_{n}
		\isoto
		\mathopen[X/{G} \mathclose]_{n}
	\]
	given by
	\[
		\bigl(h_{1},\dots ,h_{n},(n_{1},\dots ,n_{n},x ) \bigr)
		\longmapsto
		\bigl(\varphi_{h_{1} \cdots h_{n}} (n_{1} ) h_{1},\varphi_{h_{2} \cdots h_{n}} (n_{2} ) h_{2},\ldots ,\varphi_{h_{n}} (n_{n} ) h_{n},x \bigr).
	\]
	Taking colimits on both sides and applying the two lemmas yields the desired result.
\end{example}

\begingroup

\section{Homotopy limits in dg-categories}

We let~\( \categoryformat{dgSch} (k ) \) be the category of \textdef{dg-schemes} over a field~\( k \),
which for us is the opposite category of the
category~\(
	\categoryformat{dgAlg}^{\le 0} (k )
\)
of graded commutative
dg-algebras over~\( k \) sitting in non-positive degree.
If~\( X \in \categoryformat{dgSch} (k ) \), we denote by~\( A_{X} \in \categoryformat{dgAlg}^{\le 0} (k ) \)
the associated dg-algebra and call~\(
	X^{\circ}
	=
	\operatorname{Spec} (H^{0} A_{X} )
\)
the \textdef{underlying scheme} of~\( X \).
We shall often write~\( X = (X^{\circ},A_{X} ) \)
and think of~\( A_{X} \) as a dg-algebra over~\(
	\mathscr{O}_{X^{\circ}}
\).
Notice that a morphism~\(
	f
	\colon
	(X^{\circ},A_{X} )
	\to
	(Y^{\circ},A_{Y} )
\)
of affine dg-schemes is equivalent to the data of a morphism
of schemes~\(
	f
	\colon
	X^{\circ}
	\to
	Y^{\circ}
\)
and a comorphism~\(
	f^{\#}
	\colon
	A_{Y}
	\to
	A_{X}
\)
of \( \mathscr{O}_{Y^{\circ}} \)-dg-algebras.
The dg-category of \textdef{quasi-coherent sheaves} on the
dg-scheme~\( (X^{\circ},A_{X} ) \) is
just~\( \categoryformat{QCoh} (X ) = \categoryformat{$A_{X}$-dgmod} \), the category
of dg-modules over~\( A_{X} \).

An (affine) \textdef{group dg-scheme}
is a group object~\( G = (G^{\circ},A_{G} ) \)
in the category of affine dg-schemes.
This means that the underlying scheme~\( G^{\circ} \)
is an affine group scheme, and that the
comorphism~\(
	A_{G}
	\to
	A_{G}
	\mathbin{\otimes_{\mathscr{O}_{G}}}
	A_{G}
\)
of the composition map equips~\( A_{G} \) with the structure
of a Hopf dg-algebra over~\( \mathscr{O}_{G^{\circ}} \).

	Suppose that \( G \) is a group dg-scheme acting on a
	dg-scheme~\( X \).
	Then the category of~\( G \)-equivariant
	sheaves on~\( X \) is defined to be the
	homotopy limit
	\[
		\categoryformat{QCoh} (X )^{G}
		=
		\invlimformat{holim}_{\mathopen[n \mathclose] \in \mathbf{\Delta}} \categoryformat{QCoh} (X_{n} )
	\]
	taken in the category~\( \categoryformat{dgCat} (k ) \) of dg-categories over~\( k \)
	equipped with Tabuada's model structure
	\parencite[see]{tabuada},
	and where \( X_{\smallbullet} \)~is the classifying space of the action
	groupoid~\( G \times X \rightrightarrows X \).
	Similarly, define
	\[
		\smash{\bigl(\mathscr{D} \categoryformat{QCoh} (X ) \bigr)}^{G}
		=
		\invlimformat{holim}_{\mathopen[n \mathclose] \in \mathbf{\Delta}} \mathscr{D} \categoryformat{QCoh} (X_{n} ) .
	\]
	
	Suppose that \( N \) and~\( H \)~are group dg-schemes
	such that \( N \)~acts on~\( H \) by means
	of automorphisms. Then we can form the external semidirect
	product~\( N \rtimes H \) whose underlying
	dg-scheme is~\( N \times H \), and with group
	structure defined by the usual formula.
	We also call a group dg-scheme~\( G \)
	a semidirect product of group dg-subschemes~\( N \) and~\( H \)
	if it is isomorphic to the external semidirect product,
	and we write~\( G = N \rtimes H \).
	This is equivalent to having a short exact sequence of group
	dg-schemes~\(
		1 \to N \to G \to H \to 1
	\)
	such that there exists a monomorphism~\( H \into G \)
	with~\( H \into G \to H \) being the identity.
	Applying~\( H^{0} \), we get that \(
		H^{0} (A_{H} ) \to H^{0} (A_{G} ) \to H^{0} (A_{H} )
	\)~is also the identity,
	so \( H^{0} (A_{H} ) \to H^{0} (A_{G} ) \)~is injective,
	which implies that we also have a short exact sequence~\(
		1 \to N^{\circ} \to G^{\circ} \to H^{\circ} \to 1
	\)
	of the underlying group schemes.
	In particular, the underlying group scheme of~\( G \) is a semidirect product~\(
		G^{\circ} = N^{\circ} \rtimes H^{\circ}
	\).

	The same proof as in the classical case shows that we have an isomorphism
	of dg-schemes~\( G \cong N \times H \).
	This yields an isomorphism~\( A_{G} \cong A_{N} \mathbin{\otimes} A_{H} \)
	of dg-algebras over the underlying field.
	Therefore, if \( G \)~is acting on a
	dg-scheme~\( X \),
	we obtain similarly to~\cref{ex:semidirect_quotient}
	an isomorphism~\(
		\smash{A}_{G}^{\otimes n}
		\mathbin{\otimes}
		A_{X}
		\cong
		\smash{A}_{N}^{\otimes n}
		\mathbin{\otimes}
		\smash{A}_{H}^{\otimes n}
		\mathbin{\otimes}
		A_{X}
	\).
	Applying~\( \categoryformat{QCoh} \) and taking homotopy limits, we obtain
	from \cref{res:fubini_holim,res:delta_sifted}
	that
\begin{proposition}\label{ex:semidirect_equivariant_sheaves}
	If a group dg-scheme~\( G = N \rtimes H \)
	acts on an affine dg-scheme~\( X \), we have
	\[
		\categoryformat{QCoh} (X )^{G}
		\cong
		\smash{\bigl(\categoryformat{QCoh} (X )^{N} \bigr)}^{H} ,
	\]
	i.e.~you can impose \( G \)-equivariance by first imposing~\( N \)-equivariance and then \( H \)-equivariance.
	Similarly, we have for derived categories that
	\[
		\smash{\bigl(\mathscr{D} \categoryformat{QCoh} (X ) \bigr)}^{G}
		\cong
		\smash{\bigl(\smash{\bigl(\mathscr{D} \categoryformat{QCoh} (X ) \bigr)}^{N} \bigr)}^{H} .
	\]
\end{proposition}
%\end{example}

We recall from \textcite{comodules} the statement

\begin{theorem}\label{res:comodules}
	Suppose that
	\( X_{1} \rightrightarrows X_{0} \)
	is a groupoid in affine dg-schemes,
	and consider the associated classifying space~\( X_{\smallbullet} \) given by
	\[\textstyle
		X_{n}
		=
		X_{1}
		\mathbin{\times_{X_{0}}}
		X_{1}
		\mathbin{\times_{X_{0}}}
		\cdots
		\mathbin{\times_{X_{0}}}
		X_{1}
		.
	\]
	Write~\( A^{n} = \smash{A_{X_{n}}} \)
	for the associated cosimplicial system of dg-algebras.
	Let~\( A = \smash{A}^{0} \)
	and~\( C = \smash{A}^{1} \),
	and note that \( C \)~is a
	counital
	coalgebra in~\( \categoryformat{$A$-dgmod-$A$} \)
	via the map~\(
		\Delta 
		=
		\partial_{1}^{\#}
		\colon
		C
		\to
		C\mathbin{\otimes_{A}}C
	\).
	Then we have a quasi-equivalence of dg-categories
	\[\textstyle
		\invlimformat{holim}_{\mathbf{\Delta}} \categoryformat{QCoh} (X_{\smallbullet} )
		\cong
		\categoryformat{$C $-comod}_{\infty}^{\categoryformat{hcu},\categoryformat{formal}} (\categoryformat{$A $-dgmod} )
		,
	\]
	where the right-hand side denotes the dg-category
	of formal, homotopy-counital \ainfty-comodules
	over~\( C \)
	in~\( \categoryformat{$A $-dgmod} \).
\end{theorem}

\begin{remark}\label{rem:holim_of_qf}
	Suppose that we replace \( \categoryformat{QCoh} (X_{\smallbullet} ) \)
	by~\( \mathscr{D} \categoryformat{QCoh} (X_{\smallbullet} ) \).
	We may
	realize this as~\(
		\mathscr{D} \categoryformat{QCoh} (X_{\smallbullet} )
		=
		\categoryformat{QF} (\categoryformat{QCoh} (X_{\smallbullet} ) )
	\).
	Since the pullbacks are exact and have exact right adjoints, they transform
	quasi-free objects into quasi-free objects.
	This gives a direct description of the derived functors.
	Therefore,
	one may repeat the proof of the theorem above to obtain
	the realization
	\begin{align*}\textstyle
		\invlimformat{holim}_{\mathbf{\Delta}} \mathscr{D} \categoryformat{QCoh} (X_{\smallbullet} )
		&
		\cong
		\categoryformat{$C $-comod}_{\infty}^{\categoryformat{hcu},\categoryformat{formal}} (\categoryformat{QF} (\categoryformat{$\mathscr{A} $-dgmod} ) )
	\end{align*}
	for the homotopy limit of the cosimplicial system of derived dg-categories.
	In particular, derived dg-categories commute with equivariance.
\end{remark}

\endgroup

\chapter{Equivariant sheaves on loop spaces}

Let~\( X \) ba an affine, smooth, and Noetherian scheme over~\( \mathbb{C} \).
Inspired by the
Hochschild--Kostant--Rosenberg theorem, we define the
\textdef{loop space} of~\( X \) as the
affine dg-scheme~\(
	LX 
	=
	(X ,\Omega_{X} )
\),
where \( \Omega_{X} \)~is the non-positively graded algebra of differential forms
regarded as a dg-algebra with differential~\( d = 0 \).
If \( x \in X \)~is a point, the \textdef{based loop space} at~\( x \)
is the dg-scheme~\(
	L_{x} X 
	=
	LX  \mathbin{\times_{X}} x
	=
	(X ,\operatorname{Sym} (T_{x}^{*} X^{\circ} [{1}] ) )
\),
the fibre at~\( x \) of the evaluation map~~\( LX  \to X \).

\begin{example}\label{ex:based_loop_decomposition}
	In the case of a group dg-scheme~\( G \),
	we have~\(
		LG 
		=
		L_{e} G 
		\rtimes
		G
	\),
	and the based loop space is~\(
		L_{e} G 
		=
		\mathfrak{g} [{-1}]
		=
		(\operatorname{Spec} (\mathbb{C} ),\operatorname{Sym} (\mathfrak{g}^{*} [{1}] ) )
	\).
\end{example}

\begin{theorem}\label{res:main_theorem}
	Let~\( X \) ba an affine, smooth, and Noetherian scheme over~\( \mathbb{C} \) acted on by an affine group scheme~\( G \).
	The derived category of \( (G ,\Omega_{G} ) \)-equivariant \( \Omega_{X} \)-dg-modules is equivalent to the triangulated subcategory
	\[
		\langle \mathscr{O}_{X} \rangle
		\subset
		\mathscr{D}
		\bigl(
			\operatorname{Sym}_{\mathscr{O}_{X}} (\mathfrak{g} \mathbin{\otimes} \mathscr{O}_{X} [{-1}] \to{T_{X} [{-2}]} )
			\categoryformat{-dgmod}
		\bigr)^{G}
	\]
	of the derived category of \( G \)-equivariant dg-modules
	generated by \( \mathscr{O}_{X} \) and closed under small coproducts.
\end{theorem}

\begin{proof}
	Because of the decomposition~\(
		LG 
		=
		L_{e} G  \rtimes G
	\) from~\cref{ex:based_loop_decomposition},
	we recall from~\cref{ex:semidirect_equivariant_sheaves} that
	\[
		\smash{\bigl(\mathscr{D} \categoryformat{QCoh} (LX  ) \bigr)}^{LG }
		\cong
		\smash{\bigl(\smash{\bigl(\mathscr{D} \categoryformat{QCoh} (LX  ) \bigr)}^{L_{e} G } \bigr)}^{G}
		.
	\]
	Via \cref{rem:holim_of_qf}, we obtain
	\begin{align*}
		\smash{\bigl(\mathscr{D} \categoryformat{QCoh} (LX  ) \bigr)}^{LG }
			&\cong
			\invlimformat{holim}_{\mathopen[n\mathclose] \in \mathbf{\Delta}} \invlimformat{holim}_{\mathopen[m\mathclose] \in \mathbf{\Delta}} \mathscr{D} (\categoryformat{$\mathscr{O}_{G^{n}} \mathbin{\otimes} \operatorname{Sym} (\mathfrak{g}^{*} [{1}] )^{\otimes m} \mathbin{\otimes} \Omega_{X}$-dgmod} )
	\\
			&\cong
			\invlimformat{holim}_{\mathopen[n\mathclose] \in \mathbf{\Delta}} \bigl(\categoryformat{$C_{n}$-comod}_{\infty}^{\categoryformat{hcu},\categoryformat{formal}} (\categoryformat{QF} (\categoryformat{$R_{n}$-dgmod} ) ) \bigr)
	\end{align*}
	where
	\(
		R_{n} = \mathscr{O}_{G^{n}} \mathbin{\otimes} \Omega_{X}
	\)
	and
	\(
		C_{n}
		=
		\mathscr{O}_{G^{n}}
		\mathbin{\otimes}
		\operatorname{Sym} (\mathfrak{g}^{*} [{1}] )
		\mathbin{\otimes}
		\Omega_{X}
	\).
	The \( R_{n} \)--\( R_{n} \)-bimodule structure on~\( C_{n} \)
	comes from the two face maps
	\[\begin{tikzcd}
		G^{n} \times L_{e} G  \times LX 
		\ar[dr,"{\partial_{0}}"']
			&
				& G^{n} \times L_{e} G  \times LX 
				\ar[dl,"{\partial_{1}}"]
	\\
			& G^{n} \times LX 
	\end{tikzcd}\]
	given by
	\(
		\partial_{0} (g ,\gamma  ,\delta  )
		=
		(g ,\delta  )
	\)
	and~\(
		\partial_{1} (g ,\gamma  ,\delta  )
		=
		(g ,\gamma (\delta  ) )
	\).
	Taking comorphisms, the bimodule structure maps are given by
	\[\begin{tikzcd}
		\mathscr{O}_{G^{n}}
		\mathbin{\otimes}
		\operatorname{Sym} (\mathfrak{g}^{*} [{1}] )
		\mathbin{\otimes}
		\Omega_{X}
		\ar[
			dr,
			"{\partial_{0}^{\#}}"',
			<-,
			end anchor={[yshift=.7em,xshift=-2em]},
		]
			&
				& \mathscr{O}_{G^{n}}
				\mathbin{\otimes}
				\operatorname{Sym} (\mathfrak{g}^{*} [{1}] )
				\mathbin{\otimes}
				\Omega_{X}
				\ar[
					dl,
					"{\partial_{1}^{\#}}",
					<-,
					end anchor={[yshift=.7em,xshift=2em]},
				]
	\\
			&
			\mathclap{
				\mathscr{O}_{G^{n}}
				\mathbin{\otimes}
				\Omega_{X}
			}
	\end{tikzcd}\]
	where \(
		\partial_{0}^{\#}
		=
		\mathup{id}_{\mathscr{O}_{G^{n}}}
		\mathbin{\otimes}
		1
		\mathbin{\otimes}
		\mathup{id}_{\Omega_{X}}
	\)~provides the right module structure,
	whereas
	\(
		\partial_{1}^{\#}
		=
		\mathup{id}_{\mathscr{O}_{G^{n}}}
		\mathbin{\otimes}
		\mathup{ca}
	\)~provides the left module structure;
	here,
	\[
		\mathup{ca}
		\colon
		\Omega_{X}
		\to
		\operatorname{Sym} (\mathfrak{g}^{*} [{1}] )
		\mathbin{\otimes}
		\Omega_{X}
	\]
	denotes the coaction map.
	This coaction map can be made explicit:
	The action of~\( G \) on~\( X \)
	yields a map
	\[
		(\mathfrak{g} \mathbin{\otimes} \mathscr{O}_{X} )
		\oplus
		TX 
		\longto
		TX 
	\]
	given fibrewise by
	\(
		d\rho _{e ,x}
		\colon
		\mathfrak{g}
		\oplus
		T_{x} X 
		\longto
		T_{x} X 
	\),
	\(
		(v ,\xi  )
		\longmapsto
		\xi  + d(\rho (x ) )_{e} (v )
	\),
	where \( \rho (x ) \colon G \to X \)
	is the map~\( g \mapsto g x \).
	This may be dualized to a map
	\[
		T^{*} X 
		\longto
		(\mathfrak{g}^{*} \mathbin{\otimes} \mathscr{O}_{X} )
		\oplus
		T^{*} X ,
		\qquad
		\omega 
		\longmapsto
		\varphi (\omega  ) + \omega ,
		\qquad
		\text{ where }
		\varphi 
		\colon
		T^{*} X 
		\longto
		\mathfrak{g}^{*} \mathbin{\otimes} \mathscr{O}_{X}
		.
	\]
	Denoting~\( \operatorname{Sym} (\varphi  ) \) also by~\( \varphi  \),
	this provides us with the map
	\[
		\mathup{ca}
		\colon
		\operatorname{Sym}_{\mathscr{O}_{X}} (T^{*} X [{1}] )
		\longto
		\operatorname{Sym}_{\mathscr{O}_{X}} (\mathfrak{g}^{*} \mathbin{\otimes} \mathscr{O}_{X} [{1}] )
		\mathbin{\otimes_{\mathscr{O}_{X}}}
		\operatorname{Sym}_{\mathscr{O}_{X}} (T^{*} X [{1}] )
	\]
	which is our coaction map
	\[
		\mathup{ca}
		\colon
		\Omega_{X}
		\longto
		\operatorname{Sym} (\mathfrak{g}^{*} [{1}] )
		\mathbin{\otimes}
		\Omega_{X},
		\qquad
		\omega 
		\longmapsto
		1 \mathbin{\otimes} \omega 
		+
		\varphi (\omega  ) \mathbin{\otimes} 1
		.
	\]
	
	The coalgebra structure on~\( C_{n} \) comes from the composition
	\[
		\partial_{1}
		\colon
		G^{n} \times L_{e} G  \times LX 
		\mathbin{\times_{G^{n} \times LX }}
		G^{n} \times L_{e} G  \times LX 
		\longto
		G^{n} \times L_{e} G  \times LX ,
	\]
	taking a pair~\(
		((g ,\gamma_{2},\delta_{2} ),(g ,\gamma_{1},\delta_{1} ) )
	\)
	with~\( \delta_{2} = \gamma_{1} (\delta_{1} ) \)
	to~\(
		(g ,\gamma_{2} \gamma_{1},\delta_{1} )
	\).
	In other words,
	the coalgebra map~\(
		\Delta 
		=
		\partial_{1}^{\#}
		\colon
		C_{n}
		\to
		C_{n} \mathbin{\otimes_{R_{n}}} C_{n}
	\)
	is
	given by
	\[\begin{tikzcd}
			\mathllap{ C_{n} = {} }
			\mathscr{O}_{G^{n}}
			\mathbin{\otimes}
			\operatorname{Sym} (\mathfrak{g}^{*} [{1}] )
			\mathbin{\otimes}
			\Omega_{X}
			\ar[d]
				& x \mathbin{\otimes} y \mathbin{\otimes} z
				\ar[d,mapsto]
	\\
			\mathscr{O}_{G^{n}}
			\mathbin{\otimes}
			\operatorname{Sym} (\mathfrak{g}^{*} [{1}] )
			\mathbin{\otimes}
			\operatorname{Sym} (\mathfrak{g}^{*} [{1}] )
			\mathbin{\otimes}
			\Omega_{X}
			\ar[d,equal]
				& x \mathbin{\otimes} ( 1 \mathbin{\otimes} y + y \mathbin{\otimes} 1 ) \mathbin{\otimes} z
				\ar[d,mapsto]
	\\
			C_{n} \mathbin{\otimes_{R_{n}}} C_{n}
				& ( 1 \mathbin{\otimes} 1 \mathbin{\otimes} 1 ) \mathbin{\otimes} ( x \mathbin{\otimes} y \mathbin{\otimes} z )
	\\[-2em]
				& {} + ( 1 \mathbin{\otimes} y \mathbin{\otimes} 1 ) \mathbin{\otimes} ( x \mathbin{\otimes} 1 \mathbin{\otimes} z )
	\invdot\end{tikzcd}\]

	Now the coalgebra~\( C_{n} \)
	is free of finite rank as a right dg-module over the dg-algebra~\( R_{n} \).
	Therefore, we may consider the algebra
	\(
		A_{n}
		= \operatorname{Hom}_{\categoryformat{mod-$R_{n}$}} (C_{n},R_{n} )
	\)
	in the category~\(
		\categoryformat{$ { \mathscr{O}_{G^{n}} \mathbin{\otimes} \Omega_{X} } $-dgmod-$ { \mathscr{O}_{G^{n}} \mathbin{\otimes} \Omega_{X} } $}
	\).
	By adjunction, we get a quasi-equivalence of dg-categories
	\[
		\categoryformat{$C_{n}$-comod}_{\infty}^{\categoryformat{hcu},\categoryformat{formal}} (\categoryformat{QF} (\categoryformat{$R_{n}$-dgmod} ) )
		\cong
		\categoryformat{$A_{n}$-mod}_{\infty}^{\categoryformat{hu}} (\categoryformat{QF} (\categoryformat{$R_{n}$-dgmod} ) )
	\]
	between~\( \categoryformat{$C_{n}$-comod}_{\infty}^{\categoryformat{hcu},\categoryformat{formal}} \) and the category
	of homotopy-unital \ainfty-modules over~\( A_{n} \).
	The \enquote{formal} attribute becomes redundant in this case, as \ainfty-modules have no similar convergence condition.
	We then apply \cref{res:der_as_A_infty_modules,res:A_infty_hu_is_A_infty} to obtain that the right-hand side is a presentation of the derived
	dg-category~\( \mathscr{D} (\categoryformat{$R_{n}$-dgmod} ) \).
	We sum up the conclusion so far:
	\begin{lemma}
		\(
			(\mathscr{D} \categoryformat{QCoh} (LX  ) )^{LG }
			\cong
			\invlimformat{holim}_{\mathopen[n\mathclose] \in \mathbf{\Delta}} \mathscr{D} (\categoryformat{$A_{n}$-dgmod} )
		\).
	\end{lemma}

	We claim that the algebra~\( A_{n} \) has a a description as a Heisenberg algebra:
	
	\begin{lemma}
		The algebra~\( A_{n} \) is the Heisenberg algebra
		\[
			A_{n}
			=
			\operatorname{Heis}_{\mathscr{O}_{{G^{n}} \times X}} ({M_{n} [{1}]} \oplus{N_{n} [{-1}]} )
		\]
		where
		\[
			M_{n}
			=
			\mathscr{O}_{G^{n}}
			\mathbin{\otimes}
			T^{*} X 
			\qquad\text{and}\qquad
			N_{n}
			=
			\mathscr{O}_{{G^{n}} \times X}
			\mathbin{\otimes}
			\mathfrak{g}
		\]
		and the pairing is induced by the map~\(
			\varphi 
			\colon
			M_{n}
			\to
			N_{n}^{*}
		\)
		from above.
	\end{lemma}

	\begin{proof}
		The underlying complex of~\( A_{n} \) is
		\[
			A_{n}
			=
			\mathscr{O}_{G^{n}}
			\mathbin{\otimes}
			\operatorname{Sym} (\mathfrak{g} [{-1}] )
			\mathbin{\otimes}
			\Omega_{X}
			.
		\]
		For an~\( a \in A_{n} \),
		we have, by the comodule structure on~\( C_{n} \),
		that~\(
			a(x \mathbin{\otimes} y \mathbin{\otimes} z )
			=
			a(1\mathbin{\otimes} y \mathbin{\otimes} 1) \cdot x \mathbin{\otimes} z
		\)
		for all elements~\( x \mathbin{\otimes} y \mathbin{\otimes} z \in C_{n} \).
		Therefore, \( a \)~is determined by the
		element~\( a(1\mathbin{\otimes} y \mathbin{\otimes} 1) \)
		for~\( y \in \operatorname{Sym} (\mathfrak{g}^{*} [{1}] ) \). By the above coalgebra
		structure, we have for~\( a , b \in A_{n} \) that
		\[
			(a \cdot b )(1\mathbin{\otimes} y \mathbin{\otimes} 1)
			=
			a(1\mathbin{\otimes} 1\mathbin{\otimes} 1) b(1\mathbin{\otimes} y \mathbin{\otimes} 1)
			+
			a(1\mathbin{\otimes} y \mathbin{\otimes} 1) b(1\mathbin{\otimes} 1\mathbin{\otimes} 1)
		\]
		so~\( a \cdot b = a \wedge b \), and we recover
		the free multiplication on~\( \operatorname{Sym} (\mathfrak{g} [{-1}] ) \).

		The left \( R_{n} \)-multiplication on~\( A_{n} \) is given
		by the left \( R_{n} \)-multiplication on~\( R_{n} \).
		The right \( R_{n} \)-multiplication is given by the left
		\( R_{n} \)-multiplication on~\( C_{n} \).
		To rewrite this, let \(
			r
			\in
			\mathscr{O}_{G^{n}}
			\mathbin{\otimes}
			T^{*} X [{1}]
			\subset
			R_{n}
		\)
		and~\(
			a
			\in
			\mathscr{O}_{G^{n}}
			\mathbin{\otimes}
			\mathfrak{g} [{-1}]
			\mathbin{\otimes}
			1
			\subset
			A_{n}
		\).
		Then the map~\(
			a
			\colon
			\mathscr{O}_{G^{n}}
			\mathbin{\otimes}
			\mathfrak{g}^{*} [{1}]
			\mathbin{\otimes}
			\Omega_{X}
			\to
			\mathscr{O}_{G^{n}}
			\mathbin{\otimes}
			\Omega_{X}
		\)
		kills all elements except those of the form~\( x \mathbin{\otimes} y \mathbin{\otimes} z \)
		for~\(
			y
			\in
			\mathfrak{g}^{*} [{1}]
			\subset
			\operatorname{Sym} (\mathfrak{g}^{*} [{1}] )
		\).
		The element~\( \varphi (r ) \)
		is of this form.
		It follows that
		we have
		\begin{align*}
			a \cdot r 
				&=
				a(\mathup{ca} (r ) \cdot \slot )
					=
					a\bigl((r +\varphi (r ) ) \cdot \slot \bigr)
					=
					(-1)^{ \rvert r \mathclose \rvert  \rvert \slot \mathclose \rvert  }
					a(\slot \cdot r )
					+
					\langle \varphi (r ),a \rangle 
		\\
			&
				=
				(-1)^{ \rvert r \mathclose \rvert  \rvert \slot \mathclose \rvert  }
				a(\slot ) \cdot r
				+
				\langle \varphi (r ),a \rangle 
					=
					(-1)^{ \rvert r \mathclose \rvert \rvert \slot \mathclose \rvert  }
					(-1)^{ \rvert r \mathclose \rvert  ( \rvert a \mathclose \rvert  + \rvert \slot \mathclose \rvert  ) }
					r \cdot a(\slot )
					+
					\langle \varphi (r ),a \rangle 
		\\
			&
			=
			(-1)^{\rvert a \mathclose \rvert \rvert r \mathclose \rvert }
			r\cdot a
			+
			\langle \varphi (r ),a \rangle 
		\end{align*}
		(note that \( \rvert a \mathclose \rvert =1 \) and~\( \rvert r \mathclose \rvert  = -1 \), but
		they have been kept in the equation for clarity).
		In other words, we obtain the desired Heisenberg algebra structure.
	\end{proof}
	
	Applying the Koszul duality statement of \cref{res:Koszul_duality}
	over the base~\( S_{n} = \operatorname{Sym}_{\mathscr{O}_{{G^{n}} \times X}} (N[{-1}] ) \),
	we obtain a quasi-equivalence of dg-categories
	\[
		\mathscr{D} (\categoryformat{$A_{n}$-dgmod} )
		\cong
		\mathscr{D} (\categoryformat{$A_{n}^{!}$-dgmod} )_{\categoryformat{tors}}
	\]
	where
	\begin{align*}
		A_{n}^{!}
			&= \operatorname{Sym}_{\mathscr{O}_{{G^{n}} \times X}} \bigl({N_{n} [{-1}]} \xrightarrow{\varphi^{*}}{M_{n}^{*} [{-2}]} \bigr)
	\\
			&= \operatorname{Sym}_{\mathscr{O}_{{G^{n}} \times X}} \bigl(\mathscr{O}_{G^{n}} \mathbin{\otimes} \mathfrak{g} \mathbin{\otimes}{\mathscr{O}_{X} [{-1}]} \xrightarrow{\varphi^{*}} \mathscr{O}_{G^{n}} \mathbin{\otimes} TX  [{-2}] \bigr)
			.
	\end{align*}
	Thus
	\begin{align*}
	\smash{\bigl(\mathscr{D} \categoryformat{QCoh} (LX  ) \bigr)}^{LG }
		&\cong
		\invlimformat{holim}_{\mathopen[n\mathclose] \in \mathbf{\Delta}} \mathscr{D} (\categoryformat{$A_{n}^{!}$-dgmod} )_{\categoryformat{tors}}
	\\
		&\cong
			\mathscr{D} \bigl(\categoryformat{$\operatorname{Sym}_{\mathscr{O}_{X}} (\mathfrak{g} \mathbin{\otimes}{\mathscr{O}_{X} [{-1}]} \xrightarrow{\varphi^{*}} TX [{-2}] )$-dgmod} \bigr)_{\categoryformat{tors}}^{G}
	\end{align*}
	which is what we wanted to prove.
\end{proof}

\chapter*{Acknowledgements}

Part of this work was carried out under the
support of
Israel Science Foundation Grant 786/19 of Professor Vladimir Hinich.

\setcounter{biburllcpenalty}{7000}
\setcounter{biburlucpenalty}{8000}

\printbibliography

\end{document}